\pdfoutput=1
\RequirePackage{ifpdf}
\ifpdf
\documentclass[pdftex]{sigma}
\else
\documentclass{sigma}
\fi

\usepackage[all]{xy}
\usepackage[table]{xcolor}
\usepackage{euscript,mathrsfs}

\numberwithin{equation}{section}

\newtheorem{Theorem}{Theorem}[section]
\newtheorem{Corollary}[Theorem]{Corollary}
\newtheorem{Lemma}[Theorem]{Lemma}
\newtheorem{Proposition}[Theorem]{Proposition}
{\theoremstyle{definition}

\newtheorem{Remark}[Theorem]{Remark}}

\DeclareMathOperator{\sign}{sign}

\def\Q{\mathbb Q}
\def\C{\mathbb C}
\def\({\left(}
\def\){\right)}
\def\l{\lambda}

\def\G{\Gamma}

\def \Z{\mathbb Z}

\def\Im{\operatorname{Im}}
\def\Re{\operatorname{Re}}

\newcommand*\HYPERskip{&}
\catcode`,\active
\newcommand*\pFq{
\begingroup
\catcode`\,\active
\def ,{\HYPERskip}%
\doHyper
}
\catcode`\,12
\def\doHyper#1#2#3#4#5{%
\, _{#1}F_{#2}\left[\begin{matrix}#3 \smallskip \\ #4\end{matrix} \; ; \; #5\right]%
\endgroup
}

\begin{document}

\newcommand{\arXivNumber}{1803.06072}

\renewcommand{\thefootnote}{}

\renewcommand{\PaperNumber}{090}

\FirstPageHeading

\ShortArticleName{Computing Special $L$-Values of Certain Modular Forms with Complex Multiplication}

\ArticleName{Computing Special $\boldsymbol{L}$-Values of Certain Modular\\ Forms with Complex Multiplication\footnote{This paper is a~contribution to the Special Issue on Modular Forms and String Theory in honor of Noriko Yui. The full collection is available at \href{http://www.emis.de/journals/SIGMA/modular-forms.html}{http://www.emis.de/journals/SIGMA/modular-forms.html}}}

\Author{Wen-Ching Winnie LI~$^\dag$ and Ling LONG~$^\ddag$ and Fang-Ting TU~$^\ddag$}

\AuthorNameForHeading{W.-C.W.~Li, L.~Long and F.-T.~Tu}

\Address{$^\dag$~Department of Mathematics, Pennsylvania State University, University Park, PA 16802, USA}
\EmailD{\href{mailto:wli@math.psu.edu}{wli@math.psu.edu}}
\URLaddressD{\url{http://www.math.psu.edu/wli/}}

\Address{$^\ddag$~Department of Mathematics, Louisiana State University, Baton Rouge, LA 70803, USA}
\EmailD{\href{mailto:llong@lsu.edu}{llong@lsu.edu}, \href{mailto:tu@math.lsu.edu}{ftu@lsu.edu}}
\URLaddressD{\url{http://www.math.lsu.edu/~llong/}, \url{https://fangtingtu.weebly.com}}

\ArticleDates{Received April 03, 2018, in final form August 18, 2018; Published online August 29, 2018}

\Abstract{In this expository paper, we illustrate two explicit methods which lead to special $L$-values of certain modular forms admitting complex multiplication (CM), motivated in part by properties of $L$-functions obtained from Calabi--Yau manifolds defined over~$\Q$.}


\Classification{11F11; 11F67; 11M36; 33C05}

\begin{flushright}
\it Dedicated to Professor Noriko Yui
\end{flushright}

\section{Introduction}
In arithmetic geometry, the study of $L$-functions plays an important role as demonstrated by the well-known Birch and {Swinnerton--Dyer} (BSD) conjecture which connects the purely algebraic object, the algebraic rank of an elliptic curve $E$ defined over $\Q$, with a purely analytic object, the analytic rank of the Hasse--Weil $L$-function $L(E, s)$ of $E$ at the center of the critical strip. It further provides a detailed description of the leading coefficient of $L(E, s)$ expanded at this point. In this expository paper, we will explore how to use the theory of modular forms and hypergeometric functions to compute special $L$-values of modular forms. The cases we compute all have a common background of Calabi--Yau manifolds admitting complex multiplication (CM).

A Calabi--Yau manifold $X$ of dimension $d$ is a simply connected complex algebraic variety with trivial canonical bundle and vanishing $i$th cohomology group $H^i(X,{\mathcal O}_X)$ for $0<i<d$. When \smash{$d=1$}, $X$ is an elliptic curve; when $d=2$, $X$ is a K3 surface. A few papers of Noriko Yui, such as \cite{GY,Yui1}, are devoted to the modularity of Calabi--Yau manifolds $X$ defined over~$\Q$. More precisely, the action of the absolute Galois group $G_\Q$ of $\Q$ on the transcendental part of the middle \'etale cohomology $H^d_{\rm et}(X,\Q_\ell)$ should correspond to an automorphic representation according to Langlands' philosophy. In particular, when the Galois representation is $2$-dimensional, it should correspond to a classical modular form of weight $d+1$ so that both have the same associated $L$-functions. Recently, in~\cite{LTYZ} the second and third authors with Noriko Yui and Wadim Zudilin studied 14~truncated hypergeometric series that are related to weight-4 cuspidal eigenforms arising from rigid Calabi--Yau threefolds defined over $\Q$ by (super)congruences. One of the 14~modular forms in~\cite{LTYZ} has CM, that is, $f$ is invariant under the twist by a quadratic character associated to an imaginary quadratic extension~$K$ of~$\Q$. In this case there is an id\`ele class character $\chi = \chi(f)$ of~$K$ with algebraic type $d+1$ such that $L(f, s) = L\big(\chi, s-\frac{d}{2}\big)$. See \cite[Chapter~7, Section~4]{Li96} by the first author for more detail.

As explained in \cite{Deligne} by Deligne, for a weight-$k$ modular form $f(\tau)$, the values of the attached $L$-function $L(f,s)$ at the integers within its critical strip, namely $L(f, n+1)$ for $0\le n\le k-2$, are of special interest, called the periods of~$f$. Precisely, these values can be expressed as (cf.~\cite{Kohnen-Zagier})
\begin{gather}\label{eq:rn}
L(f,n+1)=\frac{(2\pi)^{n+1}}{n!}\int_0^\infty f({\rm i}t)t^n {\rm d}t, \qquad 0\le n\le k-2.
\end{gather}

Though literally $L(f,n+1)$ can be computed as a line integral, the computation can be unraveled as an iterated integral. We will illustrate this idea through examples. As these are periods of modular forms, the theory of modular forms plays a central role. Additionally, the theory of hypergeometric functions provides helpful perspectives for obtaining exact values. Our main results below are motivated in part
by the Clausen formula for hypergeometric functions (see~\eqref{eq:Clausen}), which expresses the square of a one-parameter family of integrals in terms of a~one-parameter family of iterated double integrals.

\begin{Theorem}\label{thm:1} Let $\eta(\tau)$ be the Dedekind eta function. Let $\psi$ be the id\`ele class character of~$\Q\big(\sqrt{-1}\big)$ such that $L(\psi, s-1/2)$ is the Hasse--Weil $L$-function of the CM elliptic curve $E_1\colon y^4+x^2=1$ of conductor~$32$. Then
\begin{gather*}2 L(\psi,{1/2})^2=L\big(\psi^2, {1}\big).\end{gather*}
In terms of cusp forms with CM by $\Q\big(\sqrt{-1}\big)$, the above identity can be restated as
\begin{gather*}2L\big(\eta(4\tau)^2\eta(8\tau)^2, 1\big)^2 = L\big(\eta(4\tau)^6, 2\big),\end{gather*}
where $\eta(4\tau)^2\eta(8\tau)^2$ is the weight-$2$ level $32$ cuspidal eigenform corresponding to $\psi$, and $\eta(4\tau)^6$ is the weight-$3$ level~$16$ cuspidal eigenform corresponding to~$\psi^2$.
\end{Theorem}

\begin{Theorem}\label{thm:2}Let $\chi$ be the id\`ele class character of $\Q\big(\sqrt{-3}\big)$ such that $L(\chi, s-1/2)$ is the Hasse--Weil $L$-function of the CM elliptic curve $E_2\colon x^3+y^3=1/4$ of conductor~$36$. Then
\begin{gather*} \frac32 L(\chi,1/2)^2=L\big(\chi^2,1\big) \qquad \text{and} \qquad {\frac 83} L(\chi,1/2)^3=L\big(\chi^3,3/2\big).\end{gather*}
These identities can be reformulated in terms of cusp forms with CM by $\Q\big(\sqrt{-3}\big)$ as
\begin{gather*} \frac{3}{2}L\big(\eta(6\tau)^4,1\big)^2= L\big(\eta(2\tau)^3\eta(6\tau)^3,2\big) \qquad \text{and} \qquad \frac 83 L\big(\eta(6\tau)^4,1\big)^3=L\big(\eta(3\tau)^8,3\big).\end{gather*} Here $\eta(6\tau)^4$ is the level~$36$ weight-$2$ cuspidal Hecke eigenform corresponding to $\chi$, $\eta(2\tau)^3\eta(6\tau)^3$ is the level $12$ weight-$3$ Hecke eigenform corresponding to $\chi^2$, and $\eta(3\tau)^8$ is the weight-$4$ Hecke eigenform of level $9$ corresponding to~$\chi^3$.
\end{Theorem}

In connection with Calabi--Yau manifolds explained above, the two weight-2 forms \linebreak $\eta(4\tau)^2\eta(8\tau)^2$ and $\eta(6\tau)^4$ come from the elliptic curves $E_1$ and $E_2$, respectively. The \smash{weight-3} form $\eta(4\tau)^6$ arises from any one of the elliptic K3 surfaces labeled by $\mathscr A$ in \cite{S-B} by Stienstra and Beukers. One of the defining equations is $y^2+\big(1-t^2\big)xy-t^2y=x^3-t^2x^2$. The relation between such a K3 surface and the elliptic curve $y^4+x^2=1$ is through the Shioda--Inose structure described in~\cite{Shioda-Inose}. In~\cite{S-B} Stienstra and Beukers showed that the monodromy group $\G_2$ of the elliptic fiberation of $\mathscr A$ is isomorphic to an index-2 subgroup of the congruence group~$\G_1(5)$. The group $\G_2$ itself is a noncongruence group. See~\cite{ALL} by Atkin and the first two authors for explicit congruence relations between the unique normalized weight-3 cusp form for~$\G_2$ and the congruence cusp form $\eta(4\tau)^6$.

Similarly, one of the algebraic varieties corresponding to the weight-3 form $\eta(2\tau)^3\eta(6\tau)^3$ is the elliptic K3 surface $y^2+\big(1-3t^2\big)xy-t^4\big(t^2-1\big)y=x^3$ (labeled by $\mathscr C$ in~\cite{S-B}). This K3 surface is related to $E_2$ also via the Shioda--Inose structure. The symmetric square $L$-function of $\eta(6 \tau)^4$ is $L\big(\eta(2\tau)^3\eta(6\tau)^3,s\big)L(\chi_{-3}, s-1)$ and similarly the symmetric square $L$-function of $\eta(4\tau)^2\eta(8\tau)^2$ is $L\big(\eta(4\tau)^6, s\big)L(\chi_{-1}, s-1)$. {Here $\chi_{-d}$ denotes the quadratic character attached to the imaginary quadratic extension $\Q\big(\sqrt{-d}\big)$ of~$\Q$.} A Calabi--Yau threefold corresponding to the weight-4 modular form $\eta(3\tau)^8$ via modularity is defined by $X_1^3+X_2^3+X_3^3+X_4^3-4X_5X_6=0$, $X_5^4+2X_6^2-3X_1X_2X_3X_4=0$, see~\cite{LTYZ}. The symmetric cube $L$-function of $\eta(6 \tau)^4$ is $L\big(\eta(3\tau)^8,s\big)L\big(\eta(6\tau)^4, s-1\big)$.

Our guiding philosophy is in line with the special values of the Riemann zeta function:
\begin{gather*} \zeta(2n)={\frac 12}\sum_{m\in \Z\setminus \{0\}}\frac1{m^{2n}}=(-1)^{n+1}\frac{B_{2n}\cdot(2\pi)^{2n}}{2(2n)!},\end{gather*}
where $B_n$ denotes the $n$th Bernoulli number. The Bernoullli numbers satisfy Kummer congruences which are important for the development of $p$-adic modular forms. This {was} extended to the ring of Gaussian integers by Hurwitz who showed that for any positive integer $k$, the numbers
\begin{gather*}
N(k):=\sum_{(m,n)\in \Z^2\setminus \{(0,0)\}}\frac{1}{\big(m+n\sqrt{-1}\big)^{4k}}=\sum_{(m,n)\in \Z^2\setminus \{(0,0)\}}\frac{\big(m-n\sqrt{-1}\big)^{4k}}{\big(m^2+n^2\big)^{4k}}
\end{gather*} satisfy \begin{gather*}N(k)=N(1)^k\cdot(\text{a rational number}).\end{gather*} See Hurwitz's paper~\cite{Hurwitz} or the excellent expository paper \cite{LMP} by Lee, R.~Murty, and Park in which they reproved Hurwitz's result using the fact that for any quadratic imaginary number~$\tau$, there is a transcendental number~$b_{\Q(\tau)}$, depending only on the field $\Q(\tau)$ (see~\eqref{eq:bk}, the Chowla--Selberg formula), such that for any integral weight-$k$ modular form $f$ with algebraic coefficients,
\begin{gather*}
f(\tau)/b_{\Q(\tau)}^{k}\in \overline \Q,
\end{gather*}(see \cite[Proposition~26]{Zagier} by Zagier).

In Hurwitz theorem, $N(k)$ is the Eisenstein series (see \cite[Section~2.2]{Zagier} by Zagier)
\begin{gather*}
G_{4k}(\tau):={\frac 12}\sum_{(m,n)\in \Z^2\setminus \{(0,0)\}}\frac{1}{(m+n\tau)^{4k}}=\zeta(4k)E_{4k},\qquad E_{4k}(\tau)=1-\frac{8k}{B_{4k}}\sum_{n=1}^\infty \frac{n^{4k-1}q^n}{1-q^n},
\end{gather*}where $q={\rm e}^{2\pi {\rm i}\tau}$, evaluated at $\tau=\sqrt{-1}$. Along the same vein, in \cite{Villegas-Zagier} Rodriguez-Villegas and Zagier consider the Hecke character $\varphi$ corresponding to an elliptic curve admitting CM by the imaginary quadratic field~$\Q\big(\sqrt{-7}\big)$ and they obtain a very nice formula relating the central values of the $L$-functions of $\varphi^{2k+1}$.

A special case of Damerell's result (see \cite{Damerell, Damerell2, Yager}) says that given a CM elliptic curve defined over $\Q$ with the corresponding id\`ele class character $\rho$, for each positive integer~$n$ there is a~rational number $C_{\rho,n}$ such that \begin{gather*}L\big(\rho^n, n/2\big)=C_{\rho,n} L(\rho, 1/2)^n.\end{gather*} It has been found that these numbers are generalizations of Bernoulli numbers with important arithmetic meanings (\cite{Coates-Wiles} by Coates and Wiles and \cite{Yager} by~Yager). See~\cite{Manin} by Manin, \cite{Shimura} by Shimura, \cite{Deligne} by Deligne, \cite {Kohnen-Zagier} by Kohnen and Zagier for some classical discussions on periods of modular forms, \cite{HaberlandI,HaberlandII,HaberlandIII} by Haberland and \cite{P-P} by Pa\c{s}ol and Popa for periods and inner products, \cite{ORS} by Ono, Rolen and Sprung and {\cite{RWZ} by Rogers, Wan, and Zucker} for more recent developments. Specifically, in \cite{RWZ}, from a different perspective the authors computed the periods $L\big(\eta(4\tau)^6,1\big)$ and $L\big(\eta(2\tau)^3\eta(6\tau)^3,2\big)$. Our two theorems determine the rational number $C_{\rho, n}$ explicitly for special choices of elliptic curves and positive integers~$n$. In the end of the paper, using CM values of Eisenstein series and modular polynomials, we compute a few more $C_{\chi,n}$ values, listed in Table~\ref{table:1}. Our approach can be adapted to compute $C_{\rho, n}$ for other id\`ele class characters $\rho$ associated to CM elliptic curves defined over~$\Q$; see \cite[p.~483]{Silverman} by Silverman for the full list of such elliptic curves. Note that in each case the class number of the corresponding CM field is~1. Also we focus on the periods whose corresponding Eisenstein series are holomorphic, but other periods can be handled similarly.

\section{Preliminaries}\label{section2}
\subsection{The Chowla--Selberg formula}
The Chowla--Selberg formula says that if $E$ is an elliptic curve whose endomorphism ring over~$\C$ is an order of an imaginary quadratic field $K=\Q\big(\sqrt{-d}\big)$ with fundamental discriminant $-d$, then all periods of $E$ are algebraic multiples of a particular transcendental number
\begin{gather}\label{eq:bk}
b_K:=\G\left(\frac 12\right) \prod_{0<a<d}\G\left(\frac ad\right)^{\frac{n\epsilon(a)}{4h_K}},
\end{gather}where $\G\(\cdot \)$ stands for the Gamma function, $n$ is the number of torsion elements in~$K$, $\epsilon$ is the primitive quadratic Dirichlet character modulo~$d$, that is, the quadratic character attached to~$K$ over~$\Q$, and $h_K$ is the class number of~$K$, see \cite{SC} by Selberg and Chowla, \cite{Gross} by Gross, or \cite[equation~(97)]{Zagier}. For example \begin{gather*}
b_{\Q(\sqrt{-4})}=\G\(\frac 12\)\frac{\G\(\frac 14\)}{\G\(\frac 34\)} \qquad \text{and}\qquad
b_{\Q(\sqrt{-3})}=\G\(\frac 12\)\(\frac{\G\(\frac 13\)}{\G\(\frac 23\)}\)^{3/2}.
\end{gather*}

\subsection{Gamma and beta functions}
The Gamma function satisfies two important properties in addition to the functional equation $\G(x+1)/\G(x)=x$ when $x$ is not a non-positive integer. See~\cite{AAR} by Andrews, Askey and Roy for details. The first one is the reflection formula: for $a\in \C$
\begin{gather}\label{eq:reflection}
\frac{1}{\G(a)\G(1-a)}=\frac{\sin(a\pi)}{\pi}.
\end{gather}For example, $\G\(\frac 12\)^2=\pi$, $\G\(\frac 14\)\G\(\frac 34\)=\sqrt{2}\pi$ so that \begin{gather*}b_{\Q(\sqrt{-4})}=\frac{\sqrt{2}}{2\pi}\G\(\frac 12\)\G\(\frac 14\)^2, \qquad \text{and} \qquad b_{\Q(\sqrt{-3})}=\(\frac{3}4\)^{3/4}\frac1 \pi\G\(\frac 13\)^3.\end{gather*}

The second one is the multiplication formula for $\G$: for integer $m\ge 1$ and $a\in \C$
\begin{gather}\label{eq:multiplication}
\G(a)\G\(a+\frac 1m\)\cdots \G\(a+\frac {m-1}m\)=\G(ma)\cdot (2\pi)^{(m-1)/2}m^{\frac 12-ma}.
\end{gather}

Now recall the beta function. For $a,b\in \C$ with positive real parts,
\begin{gather}\label{eq:beta}
B(a,b):=\int_0^1 x^{a-1}(1-x)^{b-1} {\rm d}x.
\end{gather}It is known that
\begin{gather}\label{eq:beta-gamma}B(a,b)=\frac{\G(a)\G(b)}{\G(a+b)}.\end{gather} For example, $B(1/3,1/3)=\frac{3^{1/2}}{2\pi}\G\(\frac 13\)^3$ is an algebraic multiple of $b_{\Q(\sqrt{-3})}$.

\subsection{Hypergeometric functions}

The (generalized) hypergeometric function with parameters $a_i$, $b_j$ and argument $x$ is defined by
\begin{gather*}
\pFq{n+1}{n}{a_1&a_2&\cdots& a_{n+1}}{&b_1&\cdots& b_n}{x}:=\sum_{k\ge 0}\frac{(a_1)_k\cdots (a_{n+1})_k}{(b_1)_k\cdots (b_n)_k}\frac{x^k}{k!},
\end{gather*}
where $(a)_k:=a(a+1)\cdots(a+k-1)$ is the Pochhammer symbol with the convention $(a)_0=1$, and can be written as $(a)_k=\G(a+k)/\G(a)$.
For $a_i$, $b_j$ such that the formal power series is well-defined, the radius of convergence for $x$ is typically~1. In particular,
${}_1F_0[a_1\,;\, x] 
=(1-x)^{-a_1}$. These functions can be defined recursively as follows \cite[equation~(2.2.2)]{AAR}: when $\operatorname{Re}(b_n)>\operatorname{Re}(a_{n+1})>0$,
\begin{gather}\label{eq:recursive}
\pFq{n+1}{n}{a_1&a_2&\cdots& a_{n+1}}{&b_1&\cdots& b_n}{x}\\
{} =\frac{1}{B(a_{n+1},b_n-a_{n+1})}\int_0^1 y^{a_{n+1}-1}(1-y)^{b_n-a_{n+1}-1}\pFq{n}{n-1}{a_1&a_2&\cdots& a_{n}}{&b_1&\cdots& b_{n-1}}{xy}{\rm d}y.\nonumber \end{gather}

In the classic developments, the $_2F_1$ functions play a vital role. They can be written using the above recipe in the following Euler integral formula
\begin{gather}\label{eq:2F1-integral}
\pFq{2}{1}{a&b}{&c}{x}=\frac 1{B(b,c-b)} \int_0^1 y^{b-1}(1-y)^{c-b-1}(1-xy)^{-a} {\rm d}y.
\end{gather}In this sense, we say the $_2F_1$ value corresponds to a 1-integral, or 1-period.

The Gauss summation formula \cite[Theorem 2.2.2]{AAR} says that for $a,b,c\!\in\! \mathbb C$ with \smash{$\operatorname{Re}(c\!-\!a\!-\!b){>}0$},
\begin{gather}\label{eq:gauss}
\pFq{2}{1}{a&b}{&c}{1}=\frac{\G(c)\G(c-a-b)}{\G(c-a)\G(c-b)}.
\end{gather} A virtue of such a formula is to give the precise value of the integral \eqref{eq:2F1-integral} in terms of Gamma values.

Here is a version of the Clausen formula \cite[p.~116]{AAR}: for $a,b,x\in \C$
\begin{gather}\label{eq:Clausen}
\pFq{2}{1}{a&b}{&a+b+\frac 12}{x}^2=\pFq{3}{2}{2a&2b&a+b}{&2a+2b&a+b+\frac 12}{x}
\end{gather} as long as both sides converge.
If we write both hand sides using the integral forms via \eqref{eq:recursive}, then the Clausen formula expresses the square of a certain 1-integral as an iterated 2-integral. From the classic result of Schwarz, hypergeometric functions are tied to automorphic forms of triangle groups. See \cite{Yang} by Yang for using hypergeometric functions to compute automorphic forms for genus 0 Shimura curves with three elliptic points and see \cite{Tu-Yang} by the third author and Yang for some applications.

\subsection{Modular forms}\label{ss:mf}
There are many texts on modular forms, for instance \cite{Cohen-stromberg, Diamond-Shurman, Li96, Ogg, Schoeneberg-B, Shimura-mf, Zagier}. Let $q={\rm e}^{2\pi {\rm i} \tau}$, where $\tau$ lies in the upper half complex plane~$\mathfrak H$. One of the most well-known modular forms is the discriminant modular form $\Delta(z)=\eta(\tau)^{24}$ where $\eta(\tau)=q^{1/24}\prod\limits_{n=1}^\infty (1-q^n)$ is the weight-1/2 Dedekind eta function. Many modular forms can be written as eta products or quotients.

Among all modular forms, theta functions and Eisenstein series (see \cite{Schoeneberg-B,Zagier}) play important roles. We first recall below the classic weight-1/2 Jacobi theta functions following mainly~\cite{Zagier} by Zagier:
\begin{gather*}
\theta_2(\tau):=\sum_{n\in\mathbb Z} q^{(2n+1)^2/8},\qquad
\theta_3(\tau):=\sum_{n\in\mathbb Z} q^{n^2/2}, \qquad
\theta_4(\tau):=\sum_{n\in\mathbb Z} (-1)^nq^{n^2/2},
\end{gather*} which satisfy the relation
\begin{gather*}
\theta_3^4=\theta_2^4+\theta_4^4.
\end{gather*}
They can be expressed in terms of the Dedekind eta function as
\begin{gather}\label{eq:theta-eta}\theta_2(\tau)=2\frac{\eta(2\tau)^2}{\eta(\tau)}, \qquad \theta_3(\tau)=\frac{\eta(\tau)^5}{\eta(\tau/2)^2\eta(2\tau)^2}, \qquad \theta_4(\tau)=\frac{\eta(\tau/2)^2}{\eta(\tau)}. \end{gather}
Then the modular $\l$-function is
\begin{gather*}
\lambda=\left(\frac{\theta_2}{\theta_3}\right)^4=\(\frac{\sqrt{2}\eta(\tau/2){\eta(2\tau)^2}}{\eta(\tau)^3}\)^8
=16\big(q^{1/2}-8q+44q^{3/2}-192q^2+718q^{5/2}+\cdots\big),
\end{gather*}
and
\begin{gather}\label{eq:1-lambda-eta}
1-\l =\left(\frac{\theta_4}{\theta_3}\right)^4=\(\frac{\eta(\tau/2)^2\eta(2\tau)}{\eta(\tau)^3}\)^8.
\end{gather} The lambda function $\l(\tau)$ generates the field of all {meromorphic} weight-0 modular forms for~$\G(2)$, the principal level-2 congruence subgroup. The group~$\G(2)$ has 3 cusps, 0, 1, ${\rm i}\infty$ at which the values of $\l$ are $1$, $\infty$, $0$ respectively.

Another relevant classical result is (see Borweins \cite{PiandAGM})
\begin{gather}\label{eq:theta3-2F1}
\theta_3^2=\pFq{2}{1}{\frac 12&\frac 12}{&1}{\frac{\theta_2^4}{\theta_3^4}}.
\end{gather} Thus
\begin{gather}\label{eq:theta4}
\theta_4^2=(1-\l)^{1/2}\pFq{2}{1}{\frac 12&\frac 12}{&1}{\l}.
\end{gather}

From \cite[Proposition 7]{Zagier}, the logarithmic derivative of the discriminant modular form $\Delta(\tau)=\eta(\tau)^{24}$ is
\begin{gather*}
\frac{1}{2\pi {\rm i}} \frac{{\rm d}}{{\rm d}\tau} \log \Delta(\tau)=E_2(\tau),\end{gather*} where \begin{gather*}E_2(\tau):=1-24\sum_{n=1}^\infty \frac{nq^n}{1-q^n}\end{gather*}
is the weight-2 holomorphic quasi-modular form for ${\rm SL}_2(\mathbb Z)$.

\begin{Lemma}\label{lem:lambda-diff}
\begin{gather}\label{eq:lambda-diff}
\l'(\tau)=\frac{{\rm d}\l(\tau)}{{\rm d}\tau}=2\pi {\rm i} \cdot 8\cdot\frac{\eta(\tau/2)^{16}\eta(2\tau)^{16}}{\eta(\tau)^{28}}=\pi {\rm i} \cdot \l(\tau)\cdot \theta_4^4(\tau).
\end{gather}
\end{Lemma}

\begin{proof}From Theorem 2.2(c) and Theorem 2.6(c) in \cite{JacobiAGM} by Borwein brothers and Garvin, one has the following two expressions of Jacobi theta functions in terms of~$E_2$:
\begin{gather*}
\theta_3^4(\tau)=\frac{4E_2(2\tau)-E_2(\tau/2)}{3},\qquad \theta_3^4(\tau)+\theta_2^4(\tau)=2 E_2(\tau)-E_2(\tau/2).\end{gather*}
Taking logarithmic derivative of $\l$ gives
\begin{align*}
\frac{\l'}{\l}(\tau)&=4\cdot \frac{2\pi {\rm i}}{24} \(E_2(\tau/2)+8E_2(2\tau)-6E_2(\tau)\)\\
&=\frac{2\pi {\rm i}}3\(\frac{3}2(E_2(\tau/2)-2E_2(\tau))+\( 4E_2(2\tau)-E_2(\tau/2) \)\)\\
&=\frac{2\pi {\rm i}}3\(3\theta_3^4-\frac{3}2\(\theta_2^4+ \theta_3^4\)\)(\tau)=\pi {\rm i} \theta_4^4(\tau).\tag*{\qed}
\end{align*}\renewcommand{\qed}{}
\end{proof}

Similarly, there are also weight-1 cubic theta functions as follows (see \cite{JacobiAGM})
\begin{gather*}
a(\tau):=\sum_{(n,m)\in \Z^2}q^{n^2+nm+m^2}=\frac{3\eta(3\tau)^3+\eta(\tau/3)^3}{\eta(\tau)},\\
b(\tau):=\sum_{(n,m)\in \Z^2}\zeta_3^{m-n}q^{n^2+nm+m^2}=\frac{\eta(\tau)^3}{\eta(3\tau)},\\
c(\tau):=\sum_{(n,m)\in \Z^2}q^{\(n+1/3\)^2+\(n+1/3\)\(m+1/3\)+\(m+1/3\)^2}=3\frac{\eta(3\tau)^3}{\eta(\tau)},
\end{gather*} where $\zeta_3={\rm e}^{2\pi {\rm i}/3}$. They satisfy the cubic relation
\begin{gather}\label{eq:cubic}
a^3=b^3+c^3.
\end{gather}
Parallel to \eqref{eq:theta3-2F1} is the following identity (see \cite{JacobiAGM})
\begin{gather}\label{eq:a-2F1}
\pFq{2}{1}{\frac 13&\frac 23}{&1}{\frac{c^3}{a^3}}=a.
\end{gather}

A finite index subgroup $\G$ of {${\rm SL}_2(\mathbb Z)$} acts on the upper half plane $ \mathfrak H$ via {fractional linear transformations}. Its fundamental domain can be compactified by adding a few missing points, called the cusps, to get the compact modular curve $X_\G$ for $\G$. The meromorphic modular functions for $\G$ form a field and we will denote it by $\C(X_{\G})$. If $X_\G$ has genus 0, $\C(X_{\G})$ has a~generator~$t$ over~$\C$ which plays a crucial role. For example, its derivative $t'=\frac{{\rm d}t}{{\rm d}\tau}$ is a weight-2 meromorphic modular form for~$\G$. When $t=\l(\tau)$, the explicit form of $t'$ is given in Lemma~\ref{lem:lambda-diff}. When $-I\notin \G$, by the Galois theory in the context of modular curves, a finite extension of~$\C(X_{\G})$ corresponds to a unique finite index subgroup $\G'$ of $\G$. Note that $\C(X_{\G'})$ is a simple field extension of~$\C(X_{\G})$. The ramifications of a generator can occur either along the cusps or along the elliptic points with specified degrees. Nevertheless, finding such a generator is equivalent to determining the group $\G'$. In some cases below, we start with a genus~0 group~$\G$ with~$\C(X_{\G})$ generated by~$t$ and compute an algebraic function on~$\G$ which leads to a genus~1 subgroup~$\G'$. In this case the invariant differential for $X_{\G'}$ naturally corresponds to the unique normalized weight-2 cusp form~$f$ for~$\G'$. This process may lead to an expression of $f(\tau){\rm d}\tau$ as an algebraic function~$R(t)$ times~${\rm d}t$. Combined with formula~\eqref{eq:rn}, $L(f,1)$ can be computed using \begin{gather*}\int_{t(0)}^{t({\rm i}\infty)} R(t){\rm d}t.\end{gather*} When $X_{\G'}$ has CM by an imaginary quadratic field~$K$, by the Chowla--Selberg formula~\eqref{eq:bk}, it is expected that $L(f,1)$ is an algebraic multiple of~$b_K$. This approach is closely related to~\cite{RWZ} by Rogers, Wan, and Zucker where they used elliptic integrals instead of modular forms.

We illustrate the above idea by the following example. The cubic equation \eqref{eq:cubic} leads to one natural way to parametrize the degree-3 Fermat curve $X^3+Y^3=1$, which is isomorphic to the modular curve~$X_0(27)$. Here we use the standard notation that $X_0(N)$ is the modular curve for \begin{gather*}\G_0(N)= \left \{\begin{pmatrix}a&b\\c&d\end{pmatrix}\in {{\rm SL}_2(\mathbb Z)}\colon c\equiv 0\mod N \right\}.\end{gather*} The genus 0 congruence subgroup~$\G_0(9)$ has 4 cusps: 0, 1/3, 2/3 and ${\rm i}\infty$. The modular curve~$X_0(27)$ is a 3-fold cover of $X_0(9)$. The cusps of $\G_0(9)$ and their behaviors in $\G_0(27)$ are summarized below:
\begin{gather*}
\arraycolsep=4.7pt\def\arraystretch{1.7}
\begin{array}{c||c|c|c|c}
\G_0(9)& {\rm i}\infty& 0& \frac 13& \frac 23 \\ \hline
\G_0(27)&{\rm i}\infty, \frac19, \frac 29& 0& \frac 13&\frac 23
\end{array}
\end{gather*}
This means that, as a cover of $X_0(9)$, $X_0(27)$ is totally ramified at the cusps $0$, $\frac 13$, $\frac 23$ and splits completely at $\infty$. In fact, the modular function
\begin{gather*}
\frac{a(3\tau)}{c(3\tau)}=\frac 13 \frac{{\eta(\tau)^3}}{\eta(9\tau)^3}+1=\frac13\big( q^{-1}+5q^2- 7q^5+3q^8+15 q^{11}+\cdots \big)
\end{gather*} is a generator of the field of modular functions for~$\G_0(9)$. It has a simple pole at the cusp~$\infty$, and takes values~$1$, $\zeta_3$, $\zeta_3^2$ at the cusps $0$, $1/3$, $2/3$, respectively. Then $X(\tau)=c(3\tau)/a(3\tau)$ is also a modular function of $X_0(27)$ and $\sqrt[3]{1-X(\tau)^3}$ matches exactly the ramification information of the covering map $X_0(27) \to X_0(9)$. Therefore, $Y(\tau)=\sqrt[3]{1-X(\tau)^3}$ is a modular function for $\G_0(27)$ which also generates the cubic extension $\C(X_0(27))/\C(X_0(9))$. As $X^3+Y^3=1$, this is a natural realization of the degree-3 Fermat group alluded to above. Therefore, up to scalar, ${\rm d}X/Y^2$ is the unique holomorphic differential $1$-form on the curve~$X_0(27)$. On the other hand, the Hecke eigenform $f_{27}(\tau):=\eta(3\tau)^2\eta(9\tau)^2$ defines a holomorphic differential $1$-form, $f_{27}(\tau){\rm d}\tau$, on~$X_0(27)$. Comparing the Fourier expansions of these two differential forms ${\rm d}X/Y^2$ and $f_{27}(\tau){\rm d}\tau$, one has
\begin{gather*}
\frac{{\rm d}X}{Y^2}(q)=3f_{27}(q)\frac{{\rm d}q}q.
\end{gather*}
Or equivalently,
\begin{gather*}
\frac{{\rm d}X}{Y^2}(\tau)=2\pi {\rm i} \cdot 3f_{27}(\tau){{\rm d}\tau}.
\end{gather*}

Set $s(\tau)=X(\tau/3)=c(\tau)/a(\tau)$ and $t(\tau)=y(\tau/3)=b(\tau)/a(\tau)$. Using the $a$, $b$, $c$-notation above, we observe that $t^3=1-s^3$ and
\begin{gather*}
f_{27}\(\tau/3\) = \eta(\tau)^2\eta(3\tau)^2=\frac 13 b(\tau)c(\tau)=\frac13s(\tau)t(\tau)a(\tau)^2.
\end{gather*}
It follows that
\begin{gather*}
\frac{{\rm d}s}{t^2}(\tau)=2\pi {\rm i} \cdot f_{27}(\tau/3){{\rm d}\tau}.
\end{gather*}
In other words,
\begin{gather}\label{eq:ds} \frac{{\rm d}s}{{\rm d}\tau}= 2\pi {\rm i} \cdot s\big(1-s^3\big)a^2/3.\end{gather}

\subsection{Id\`ele class characters and modular forms}\label{ss:idele}
Given a number field $K$, denote by $\Sigma(K)$ the set of places of $K$, and for each $v \in \Sigma(K)$, let~$K_v$ denote the completion of $K$ at $v$. When $v$ is a nonarchimedean place, let $\mathcal O_v$ be the ring of integers of $K_v$, $\mathcal M_v = \pi_v \mathcal O_v$ its unique maximal ideal, and $\mathcal U_v$ the group of units in $\mathcal O_v$. The group of id\`eles $I_K$ of $K$ is the restricted product $\mathop{\prod'}\limits_{v \in \Sigma(K)} K_v^\times$ with respect to $\mathcal U_v$ at the nonarchimedean places $v$ of $K$. It is a locally compact topological group in which the multiplicative group $K^\times$ is diagonally embedded.

A character $\xi$ of $I_K$ is a continuous homomorphism from $I_K$ to $\C^\times$. Denote by $\xi_v$ its restriction to $K_v^\times$. Then we can write $\xi = \prod\limits_{v \in \Sigma(K)} \xi_v$. Furthermore, owing to the topology on~$I_K$, one can show that, for almost all nonarchimedean places $v$ of $K$, the character $\xi_v$ is trivial on $\mathcal U_v$, hence it is determined by its value at any uniformizer $\pi_v$ at $v$ since $K_v^\times = \mathcal U_v \times \langle \pi_v \rangle$. In this case we say that $\xi$ and $\xi_v$ are unramified at~$v$. At a nonarchimedean place $v$ where $\xi_v$ is not trivial on~$\mathcal U_v$, again by the topology on $K_v^\times$, one can show that there is a smallest positive integer $\mathfrak f_v$ so that~$\xi_v$ is trivial on the subgroup $1 + \mathcal M_v^{\mathfrak f_v}$ of $\mathcal U_v$. In this case we say that $\xi$ and $\xi_v$ are ramified at~$v$ and the product of $\mathcal M_v^{\mathfrak f_v}$ over ramified places~$v$ is called the {\it conductor} of~$\xi$.

When a character $\xi$ of $I_K$ is trivial on $K^\times$, it is called an {\it id\`ele class character} of~$K$. The weak approximation theorem (cf.~\cite[p.~117]{Neu}) implies that an id\`ele class character $\xi = \prod_v \xi_v$ with a~given conductor is determined by the local components $\xi_v$ for all but finitely many places~$v$.

The nonarchimedean places $v$ of $K$ are in one-to-one correspondence with the maximal ideals~$\mathfrak P_v$ of the ring of integers of~$K$. Given an id\`ele class character $\xi=\prod\limits_{v \in \Sigma(K)} \xi_v$, the formula
\begin{gather*}\xi'(\mathfrak P_v) = \xi_v(\pi_v)
\end{gather*}
defines a character ${\xi'}$ on the free abelian group generated by all $\mathfrak P_v$ with $\xi_v$ unramified, in other words, the group of fractional ideals of $K$ coprime to the conductor of~$\xi$. In the literature ${\xi'}$ is called a \emph{Hecke Grossencharacter}. Upon checking the behavior of~${\xi'}$ on the integral principal ideals, one finds that ${\xi'}$ has the same conductor as $\xi$. Conversely, given a Hecke Grossencha\-rac\-ter~${\xi'}$, the above formula defines an unramified character $\xi_v$ of $K_v^\times$ for all but finitely many places $v$ of $K$, which can be uniquely extended to an id\`ele class character $\xi = \prod_v \xi_v$ with the same conductor as ${\xi'}$ by the weak approximation theorem. So the two kinds of characters are the same. The reader is referred to Section~6, Chapter~VII of the book \cite{Neu} by Neukirch for more detail.

A typical example of an unramified id\`ele class character is $\xi = |~|_K= \prod\limits_{v \in \Sigma(K)} |~|_v$, the absolute value of~$I_K$. Here $|~|_v$ is the standard valuation at $v$. More precisely, if $v$ is a real place, it is the usual absolute value on~$\mathbb R$; if $v$ is a complex place, it is the square of the usual absolute value on $\C$; if $v$ is a nonarchimedean place, it is given by
\begin{gather*} |\pi_v|_v = \frac{1}{Nv}\qquad {\rm and} \qquad |u_v|_v = 1 \qquad {\rm for} \ \ u_v \in \mathcal U_v\end{gather*}
with $Nv$ the cardinality of the residue field~$\mathcal O_v/\mathcal M_v$.

\begin{Remark}\label{rem:2.1}In \cite[Chapter~5, Proposition~1]{Li96} it is shown that any character $\xi$ of~$I_K$ can be written as the product of $|~|_K^{s'}$ for some complex number $s'$ times a unitary id\`ele class charac\-ter~$\xi_1$ of $I_K$ which takes values in the unit circle $S^1 \subset \C^\times$.
\end{Remark}

To an id\`ele class character $\xi$ of $K$, we associate an $L$-function defined as
\begin{gather}\label{eq:L-idele} L(\xi, s) = \prod_{v~{\rm nonarchimedean}\atop{\xi_v ~ {\rm unramified}}} \frac{1}{1 - \xi_v(\pi_v)(Nv)^{-s}}.\end{gather}
Note that the $L$-function attached to the trivial character is nothing but the Dedekind zeta function of $K$, which converges absolutely for $\Re (s) > 1$. The same holds for $\xi$ unitary. In general, by Remark {\ref{rem:2.1}}, we can write $\xi = |~|_K^{s'}\xi_1$ with $s' \in \C$ and $\xi_1$ a unitary id\`ele class character. Since $L(\xi, s) = L(\xi_1, s+s')$, we conclude that the $L$-function attached to $\xi$ converges absolutely to a holomorphic function on the right half-plane $\Re(s) > 1 - \Re(s')$. It suffices to understand the analytic behavior of $L$-functions attached to unitary id\`ele class characters of~$K$. This was studied by Hecke for Grossencharacters. Hecke's result was reproved in Tate's thesis~\cite{Tate} using adelic language, summarized below.

\begin{Theorem} Let $K$ be a number field with different $\mathfrak d$. Let $\xi$ be a unitary id\`ele class character of $K$ with conductor $\mathfrak f$. The associated $L$-function $L(\xi, s)$ defined above is holomorphic on $\Re(s) > 1$. It can be analytically continued to a meromorphic function on the whole $s$-plane, bounded at infinity in each vertical strip of finite width, and holomorphic everywhere except for a simple pole at $s=1$ when $\xi$ is the trivial character. Further, there is a suitable $\Gamma$-product $L_\infty(\xi_\infty, s)$, depending on $\xi_v$ at the archimedean places $v$ of $K$, such that
\begin{gather*} \Lambda(\xi, s) := L_\infty(\xi_\infty, s)L(\xi, s)\end{gather*}
satisfies the functional equation
\begin{gather}\label{eq:feq-idele} \Lambda(\xi, s) = W(\xi)N_{K/\Q}(\mathfrak {df})^{\frac12 - s}\Lambda\big(\xi^{-1}, 1-s\big),
\end{gather}
where $W(\xi)$ is a constant of absolute value $1$ and $\xi^{-1}$ is the inverse of $\xi$.
\end{Theorem}
Here $W(\xi)$, called the root number of $\xi$, is equal to the Gauss sum of $\xi$ divided by its absolute value. See \cite{Tate} for details.

In particular, when $K$ is an imaginary quadratic extension of $\Q$, it has one infinite place $\infty$ with $K_\infty = \C$. The $\Gamma$-factor $L_\infty(\xi_\infty, s)$ is equal to $(2\pi)^{-s}\Gamma(s)$ for all unitary $\xi$. Such a charac\-ter~$\xi$ is said to have \emph{algebraic type} $k$ if $\xi_\infty$ maps $z \in \C^\times$ to $\xi_\infty(z) = (z/|z|)^n$ with $|n| = k-1$. The above theorem {for} $\xi$ algebraic of type $k \ge 1$ combined with the converse theorem for modular forms proved by Weil~\cite{Weil} implies the existence of a modular form $f_{\xi} = \sum\limits_{n \ge 0} a_n q^n$ of weight $k$ such that the associated $L$-function $L(f_\xi, s) := \sum\limits_{n \ge 1} a_n n^{-s}$ satisfies the relation
\begin{gather*} L(f_\xi, s)= L\(\xi, s-\frac{k-1}{2}\).
\end{gather*}
The form $f_\xi$ is cuspidal if $\xi$ is nontrivial. See \cite[Chapter~7, Section~4]{Li96} for details. Observe that, since the $L$-function attached to $f$ is Eulerian at all primes, the form $f$ is a Hecke eigenfunction.

We end this subsection by noting that for an elliptic curve $E$ defined over $\Q$ of conductor $N$ with CM by $K$, it was known to Deuring that the Hasse--Weil $L$-function $L(E, s)$ attached to~$E$ obtained by counting $\mathbb F_p$-rational points on the reduction of~$E$ modulo the prime $p$ is equal to $L\big(|~|_K^{-1/2}\xi, s\big)= L\big(\xi, s -\frac 12\big)$ for some nontrivial unitary id\`ele class character~$\xi$ of $K$, algebraic of type $2$. The above discussion says that all positive powers of $\xi$ also correspond to modular forms.

\subsection{Eisenstein series}\label{ss:2.5} We now recall some useful facts of Eisenstein series of general level.
Define the level $N$ holomorphic Eisenstein series as follows. For details see \cite{Cohen-stromberg, Diamond-Shurman}. For a fixed pair of integers $(a_1,a_2)$, for $k\ge 3$, we let
\begin{gather*}
\mathcal G_{k,(a_1,a_2;N)}(\tau)=\sum_{(m,n)\in \Z^2\atop (m,n)\equiv (a_1,a_2) \mod N}\frac1{(m\tau+ n)^k}.
\end{gather*}
The series is a weight $k$ holomorphic modular form on
$\G(N)=\left\{\gamma\in {{\rm SL}_2(\mathbb Z)}\colon \gamma\equiv I_2 \mod N\right\}.
$ When $N\ge 3$, $\G(N)$ does not contain $-I_2$. For $k=1$ or $2$, the series does not converge absolutely so we adopt the following approach using the more general non-holomorphic Eisenstein series
\begin{gather*}
\mathcal G_{k,(a_1,a_2;N)}^\ast(s,\tau)=\sum_{(m,n)\in \Z^2\atop (m,n)\equiv (a_1,a_2) \mod N}\frac1{(m\tau+ n)^k}\frac{\Im(\tau)^s}{|m\tau+ n|^{2s}}.
\end{gather*} We now recall some basic properties of this function. For details see \cite[Proposition~5.2.2]{Cohen-stromberg} by Cohen and Str\"omberg. Firstly the series $\mathcal G_{k,(a_1,a_2;N)}^\ast(s,\tau)$ converges absolutely and uniformly on any compact subset of the upper half complex plane when $\operatorname{Re}(2s+k)>2$ thus it is continuous at $s=0$ when $k\ge 3$. Also for a fixed $\tau$, there exists a meromorphic continuation of $\mathcal G_{k,(a_1,a_2;N)}^\ast(s,\tau)$ to the whole $s$-plane which is parallel to the Fourier series stated in Proposition~\ref{prop:Lip} below. In our later application, we are interested in the following series
\begin{gather*}
\mathcal G^\ast_{k,(a_1,a_2;N)}(\tau):= \mathcal G_{k,(a_1,a_2;N)}^\ast(0,\tau)=\lim_{\operatorname{Re}(s)>0\atop s\rightarrow 0} \mathcal G_{k,(a_1,a_2;N)}^\ast(s,\tau).
\end{gather*}
Thus when $k\ge 3$, the series $\mathcal G^\ast_{k,(a_1,a_2;N)}(\tau)$ and $\mathcal G_{k,(a_1,a_2;N)}(\tau)$ coincide.
For integers $a \ge 0$, $N\ge 1$ and $k\ge 1$, we can define the series
\begin{gather*}
\mathcal G_{k,(a;N)}^\ast(\tau):= \sum_{i=0}^{N-1} \mathcal G_{k,(a, i;N)}^\ast(\tau) =\sum_{m,n\in\Z}\frac{1}{\((Nm+a)\tau+n\)^k}
\end{gather*}
formally. To give the meromorphic continuation of $\mathcal G_{k,(a_1,a_2;N)}^\ast(s,\tau)$ and obtain the Fourier expansions of $\mathcal G_{1,(a;N)}^\ast$ and $\mathcal G_{2,(a;N)}^\ast$, we use the following {Lipschitz} summation formula. See \cite[Section~3.5]{Cohen-stromberg} by Cohen and Str\"omberg, \cite[Theorem 10.4.3]{Cohen-ant} by Cohen, or \cite[Section 5.3]{Knopp} by Knopp for details.

\begin{Proposition}[{\cite[Corollary 3.5.7(a)]{Cohen-stromberg}}]\label{prop:Lip} For {$\tau\in \mathfrak H$,} $k\in \Z_{\ge 0}$, ${\Re}(s)>(1-k)/2$, we have
\begin{gather*}
\G(s+k)\sum_{n\in \Z}\frac{1}{(\tau+n)^k| \tau+n|^{2s}}=(-{\rm i})^k\sqrt \pi \frac{\G(s+(k-1)/2)\G(s+k/2)}{\G(s)}\Im(\tau)^{1-2s-k}\\
\hphantom{\G(s+k)\sum_{n\in \Z}\frac{1}{(\tau+n)^k| \tau+n|^{2s}}=}{} +(-2\pi)^k2^{s+1/2}\pi^{2s-1/2}\! \sum_{n\neq 0} \sign (n)^k|n|^{2s+k-1}W_k(2\pi n\tau,s),
\end{gather*}
where $W_k(z;s)$ is defined inductively as follows:
\begin{gather*}
W_0(z;s)=|{\Im}(z)|^{1/2-s}{\rm e}^{{\rm i}\operatorname{Re}(z)}K_{s-1/2}(|\Im(z)|),
\end{gather*}
with
\begin{gather*}
K_a(x)=\frac 12\int_0^\infty t^{a-1}{\rm e}^{-\frac x2(t+1/t)}{\rm d}t \qquad \text{for} \ x>0,
\end{gather*} being a $K$-Bessel function {\rm \cite[Definition~3.2.8]{Cohen-stromberg}} and for $k\ge 1$, $z\in \C$ with $\Im(z)\neq 0$,
\begin{gather*}
W_k(z;s)=\frac{\partial W_{k-1}(z;s)}{\partial z}.
\end{gather*}
\end{Proposition}
\begin{Proposition}[{\cite[Lemma 3.5.6(b)]{Cohen-stromberg}}] When $k>0$, $s=0$, we have
\begin{gather*}
W_k(z;0)=
\begin{cases}
{\rm i}^k\sqrt{\dfrac{\pi}{2}}{\rm e}^{{\rm i}z}, &\mbox{if } \Im(z)>0,\\
0, &\mbox{if } \Im(z)<0.
\end{cases}
\end{gather*}
\end{Proposition}

In addition, we have to deal with
\begin{gather*}
\sum_{m\equiv a \mod N} \frac 1{m^s}=\frac 1{N^s} \sum_{m\in \Z} \frac 1{\(m+\frac aN\)^s},
\end{gather*}
which is related to the Hurwitz zeta function
\begin{gather}\label{eq:Hzeta}
\zeta (x,s):=\sum _{{n=0}}^{\infty}{\frac {1}{(n+x)^{{s}}}}, \qquad \operatorname{Re}(s)>0, \qquad \operatorname{Re}(x)>0.
\end{gather}
The function $\zeta(x;s)$ has a simple pole at $s=1$ with residue $1$. And $\zeta(x;0)=1/2-x$ (see \cite[Section~9.6.1]{Cohen-ant} or \cite[Proposition~3.5.8]{Cohen-stromberg}).

The next result is about the Fourier expansion of the Eisenstein series \begin{gather*}
\mathcal G_{k,(a:N)}^\ast(s,\tau)=\sum_{m,n\in\Z}\frac{1}{\((Nm+a)\tau+n\)^k}\frac{\Im(\tau)^s}{|(Nm+a)\tau+n|^{2s}}, \qquad N,a\in \Z_{>0},
\end{gather*}when $s\rightarrow 0$.

\begin{Theorem}\label{thm:G*}We have the Fourier expansions
\begin{gather*}
\mathcal G_{1,(a;N)}^\ast(\tau)=- {\rm i}\pi\(1-\frac {2a}{N}\)-2\pi {\rm i}\sum_{n> 0}\frac{q^{na}- q^{n(N-a)}}{1-q^{Nn}},\\
\mathcal G_{2,(a;N)}^\ast(\tau)=- \frac {\pi}{N\Im(\tau)}-(2\pi)^2\sum_{n> 0}n\frac{q^{na}+q^{n(N-a)}}{1-q^{Nn}},
\end{gather*}
and for integer $k\ge 3$,
\begin{gather*}
\mathcal G_{k,(a;N)}^\ast(\tau)
=\frac{(-2\pi {\rm i})^k}{(k-1)!}\sum_{n\ge 1} n^{k-1}\(\frac{q^{na}+(-1)^kq^{n(N-a)}}{1-q^{nN}}\).
\end{gather*}
\end{Theorem}

\begin{proof}For any fixed integer $k\ge 1$ and $\operatorname{Re}(s)>0$, we have
\begin{gather*}
\mathcal G_{k,(a;N)}^\ast(s,\tau) =\sum_{m\ge 0\atop n\in\Z}\frac{1}{\((Nm+a)\tau+n\)^k}\frac{\Im(\tau)^s}{|(Nm+a)\tau+n|^{2s}}\\
\hphantom{\mathcal G_{k,(a;N)}^\ast(s,\tau) =}{} +\sum_{m\ge 0\atop n\in\Z}\frac{(-1)^k}{\((Nm+(N-a))\tau+n\)^k}\frac{\Im(\tau)^s}{|(Nm+(N-a))\tau+n|^{2s}}.
\end{gather*}
From Proposition \ref{prop:Lip}, we obtain
\begin{gather*}
\G(s+k)\sum_{m\ge 0\atop n\in\Z} \frac{1}{\((Nm+a)\tau+n\)^k}\frac{\Im(\tau)^s}{|(Nm+a)\tau+n|^{2s}}\\
\qquad{} = \sum_{m\ge 0}\Im(\tau)^s(-{\rm i})^k\sqrt \pi \frac{\G(s+(k-1)/2)\G(s+k/2)}{\G(s)}\Im((Nm+a)\tau)^{1-2s-k}\\
\qquad\quad{} +(-2\pi)^k2^{s+1/2}\pi^{2s-1/2}\sum_{m\ge 0}\Im(\tau)^s\sum_{n\neq 0}{\sign}(n)^k|n|^{2s+k-1}W_k(2\pi n(Nm+a)\tau,s)\\
\qquad{} = (-{\rm i})^k\sqrt \pi \frac{\G(s+(k-1)/2)\G(s+k/2)}{\G(s)}\Im(\tau)^{1-s-k}\sum_{m\ge 0}(Nm+a)^{1-2s-k}\\
\qquad\quad{} +(-2\pi)^k2^{s+1/2}\pi^{2s-1/2}\Im(\tau)^s\sum_{n\neq 0}{\sign}(n)^k|n|^{2s+k-1}\sum_{m\ge 0}W_k(2\pi n(Nm+a)\tau,s)\\
\qquad{} \overset{\eqref{eq:Hzeta}}= (-{\rm i})^k\sqrt \pi \frac{\G(s+(k-1)/2)\G(s+k/2)}{\G(s)}\Im(\tau)^{1-s-k}\frac 1{N^{k+2s-1}}\zeta\(a/N;k+2s-1\)\\
\qquad\quad{} +(-2\pi)^k2^{s+1/2}\pi^{2s-1/2}\Im(\tau)^s\sum_{n\neq 0}{\sign}(n)^k|n|^{2s+k-1}\sum_{m\ge 0}W_k(2\pi n(Nm+a)\tau,s)
\end{gather*}
and similarly,
\begin{gather*}
\G(s+k)\sum_{m\ge 0\atop n\in\Z}\frac{(-1)^k}{\((Nm+(N-a))\tau+n\)^k}\frac{\Im(\tau)^s}{|(Nm+(N-a))\tau+n|^{2s}}\\
\quad{} ={\rm i}^k\sqrt \pi \frac{\G(s+(k-1)/2)\G(s+k/2)}{\G(s)}\Im(\tau)^{1-s-k}\frac 1{N^{k+2s-1}}\zeta\((N-a)/N;k+2s-1\)\\
\quad\quad{} +(2\pi)^k2^{s+1/2}\pi^{2s-1/2}\Im(\tau)^s\sum_{n\neq 0}{\sign}(n)^k|n|^{2s+k-1}\sum_{m\ge 0}W_k(2\pi n(Nm+N-a)\tau,s).
\end{gather*}

Hence, when $k=2$, we have
 \begin{gather*}
\G(s+2)\mathcal G_{2,(a;N)}^\ast(s;\tau)\\
\qquad{} =-\sqrt \pi s\G\(s+\frac 12\)\Im(\tau)^{-1-s}\frac 1{N^{1+2s}}\(\zeta\(\frac aN;1+2s\)+\zeta\(\frac {N-a}N;1+2s\)\)\\
\qquad \quad{}+(2\pi)^22^{s+1/2}\pi^{2s-1/2}\Im(\tau)^s\sum_{n\neq 0}|n|^{2s+1}\sum_{m\ge 0}W_2(2\pi n(Nm+a)\tau,s)\\
\qquad\quad{} +(2\pi)^22^{s+1/2}\pi^{2s-1/2}\Im(\tau)^s\sum_{n\neq 0}|n|^{2s+1}\sum_{m\ge 0}W_2(2\pi n(Nm+N-a)\tau,s).
\end{gather*}
As $s\rightarrow 0^+$,
\begin{gather*}
\mathcal G_{2,(a;N)}^\ast(\tau) = -\pi \frac 1{\Im(\tau)}\frac 1{N} +(2\pi)^22^{1/2}\pi^{-1/2}\\
\hphantom{\mathcal G_{2,(a;N)}^\ast(\tau) =}{}\times \sum_{n> 0}n\sum_{m\ge 0}\(W_2(2\pi n(Nm+a)\tau,0)+W_2(2\pi n(Nm+N-a)\tau,0)\)\\
\hphantom{\mathcal G_{2,(a;N)}^\ast(\tau)}{} =- \frac {\pi}{\Im(\tau)}\frac 1{N}-(2\pi)^2\sum_{n> 0}n\sum_{m\ge 0}\big(q^{n(Nm+a)}+q^{n(Nm+N-a)}\big)\\
\hphantom{\mathcal G_{2,(a;N)}^\ast(\tau)}{} = - \frac {\pi}{\Im(\tau)}\frac 1{N}-(2\pi)^2\sum_{n> 0}n\frac{q^{na}+q^{n(N-a)}}{1-q^{Nn}}.
\end{gather*} To get the first equality, we use the facts that $\G\(\frac 12\)=\sqrt{\pi}$ and $\zeta(x,s)$ has a simple pole at $s=1$ with residue~1. Notice that when $n<0$, for $n' \in \{2\pi n(Nm+a)\tau, 2\pi n(Nm+N-a)\tau \}$, the term $W_2(n',0)$ is $0$ since the imaginary part of $n'$ is negative.

If $k=1$, we have
\begin{gather*}
\G(s+1)\mathcal G_{1,(a;N)}^\ast(s;\tau)=-{\rm i}\sqrt \pi \G\(s+\frac 12\)\Im(\tau)^{-s}\frac 1{N^{2s}}\(\zeta\(\frac aN;2s\)-\zeta\(\frac{N-a}N;2s\)\)\\
\qquad {} +(-2\pi)2^{s+1/2}\pi^{2s-1/2}\Im(\tau)^s\sum_{n\neq 0}{\sign}(n) n^{2s}\sum_{m\ge 0}W_1(2\pi n(Nm+a)\tau,s)\\
\qquad{} -(-2\pi)2^{s+1/2}\pi^{2s-1/2}\Im(\tau)^s\sum_{n\neq 0}{\sign}(n) n^{2s}\sum_{m\ge 0}W_1(2\pi n(Nm+N-a)\tau,s).
\end{gather*}
As $s\rightarrow 0^+$,
\begin{gather*}
\mathcal G_{1,(a;N)}^\ast(\tau) =- {\rm i}\pi\(1-\frac {2a}{N}\) +(-2\pi)2^{1/2}\pi^{-1/2}\\
\hphantom{\mathcal G_{1,(a;N)}^\ast(\tau) =}{}\times \sum_{n> 0} \sum_{m\ge 0}\(W_1(2\pi n(Nm+a)\tau,{0})- W_1(2\pi n(Nm+N-a)\tau,{0})\)\\
\hphantom{\mathcal G_{1,(a;N)}^\ast(\tau)}{}= - {\rm i}\pi\(1-\frac {2a}{N}\)-2\pi {\rm i}\sum_{n> 0}\frac{q^{na}- q^{n(N-a)}}{1-q^{Nn}}.
\end{gather*}
For the first equality, we use $\zeta(x;0)=1/2-x$.

If $k\ge 3$, when $s\rightarrow 0$, $\frac{\G(s+(k-1)/2)\G(s+k/2)}{\G(s)}\Im(\tau)^{1-s-k}\frac 1{N^{k+2s-1}}\zeta\(A/N;k+2s-1\)$ goes to 0 for both $A=a$ and $A=N-a$, the claim follows.
\end{proof}

\section{A proof of Theorem \ref{thm:1} using hypergeometric functions}
Let $x=\frac{4\eta(2\tau)^4\eta(8\tau)^8}{\eta(4\tau)^{12}}$. It is a generator of the field of the modular functions for $\G_0(8)$, a genus~0 subgroup. Note that $x^2=\l(4\tau)$. At the 4 cusps $0$, $1/2$, $1/4$, ${\rm i}\infty$ of $\G_0(8)$, it takes values $1$, $-1$, $\infty$, $0$ respectively. The genus 1 modular curve $X_0(32)$, for the congruence group $\G_0(32)$, is a~4-fold ramified cover of the modular curve $X_0(8)$, for the group $\G_0(8)$.

The cusps of $\G_0(8)$ and their behaviors in $\G_0(32)$ are summarized below:
\begin{gather*}
\arraycolsep=4.7pt\def\arraystretch{1.7}
\begin{array}{c||c|c|c|c}
\G_0(8)& {\rm i}\infty& \frac14& \frac 12& 0\\\hline
\G_0(32)&{\rm i} \infty, \frac18, \frac 38, \frac 1{16}& \frac14,\frac 34& \frac 12&0
\end{array}
\end{gather*}
This means the covering $X_0(32)$ of $X_0(8)$ ramifies completely at the cusps $\frac 12$ and~0; splits completely at the cusp~${\rm i}\infty$; and splits at the cusp~$\frac 14$ into two cusps $\frac 14$ and $\frac 34$, each with ramification degree~2. Thus
\begin{gather*}
y=\big(1-x^2\big)^{1/4}
\end{gather*} is a modular function for $\G_0(32)$. Consequently, a defining equation for the genus 1 modular curve $X_0(32)$ is
\begin{gather*}
y^4=1-x^2.
\end{gather*}
The unique up to scalar holomorphic differential $1$-form on $X_0(32)$ is given by $\frac{{\rm d}x}{y^3}$. By Lemma~\ref{lem:lambda-diff}, one has the following expression of $\frac{{\rm d}x}{y^3}$ as a function of~$\tau$:
\begin{align}
\frac{{\rm d}x(\tau)}{y(\tau)^3} &=\frac{{\rm d}x(\tau)}{(1-x(\tau)^2)^{3/4}}=\frac{{\rm d}\l(4\tau)}{2\l(4\tau)^{1/2}(1-\l(4\tau))^{3/4}}\nonumber\\
& \overset{\eqref{eq:lambda-diff}}= 2\pi {\rm i} \frac{\l(4\tau)^{1/2}\theta_4(4\tau)^4}{(1-\l(4\tau))^{3/4}}{\rm d}\tau \overset{\eqref{eq:theta-eta}, \ \eqref{eq:1-lambda-eta}}= 8 \pi {\rm i} \cdot {f_{32}(\tau)} {\rm d} \tau ,\label{eq:31}
\end{align}
where $f_{32}(\tau):=\eta(4\tau)^2\eta(8\tau)^2$ is the unique weight-$2$ level $32$ normalized cuspidal newform.

\begin{Lemma}$ L(f_{32}, 1) = L\(\eta(4\tau)^2\eta(8\tau)^2,1\)=2^{-7/2}B(1/4,1/4)=\frac18b_{\Q(\sqrt{-4})}$.
\end{Lemma}

\begin{proof}By \eqref{eq:rn},
\begin{align*}
L(f_{32},1)&= {2\pi} \int_0^{\infty} f_{32}({\rm i}t) {\rm d} (t)=-{\rm i}{2\pi} \int_0^{{\rm i}\infty} f_{32}({\rm i}t) {\rm d} ({\rm i}t)\\
&= -\frac{1}{4}\int_1^0 \frac{{\rm d}x}{(1-x^2)^{3/4}}=-\frac{1}{8} \int_{1}^0\frac{{\rm d}\l}{\l^{1/2}(1-\l)^{3/4}}\\ & \overset{ \eqref{eq:beta}}= \frac{1}{8} B(1/2,1/4)\overset{\eqref{eq:beta-gamma}}= \frac 18 \frac{\G\(\frac 12\)\G\(\frac 14\)}{\G\(\frac 34\)}\overset{\eqref{eq:reflection}}=\frac 18 \frac{\G\(\frac 12\)\G\(\frac 14\)^2 \sqrt{2}}{2\pi}= \frac18b_{\Q(\sqrt{-4})}.\tag*{\qed}
\end{align*}\renewcommand{\qed}{}
\end{proof}

Since the elliptic curve $y^4=1-x^2$ has CM by the imaginary quadratic field $K = \Q({\rm i})$, there is a character $\psi$ of the id\`ele class group of $K$, algebraic of type $2$ at the complex place~$\infty$, such that $L(f_{32}, s) = L\big(\psi, s-\frac{1}{2}\big)$. See {Section~\ref{ss:idele} or} \cite[p.~145]{Li96} by the first author for more details. The type $2$ condition (see Section~\ref{ss:idele}) means that $\psi_\infty(z) = z/|z|$ for $z \in \mathbb C^\times$. The field~$K$ has only one place $\mathfrak P$ dividing $2$ and norm $N(\mathfrak P) = 2$. As the norm of the different~$\mathfrak d$ of~$K$ over~$\Q$ is equal to~$4$, the absolute value of the discriminant of~$K$ over~$\Q$, and $32$ is equal to the norm of the product of $\mathfrak d$ and the conductor of $\psi$, we conclude that the conductor of~$\psi$ is~$\mathfrak P^3$. Write $\mathcal M_{\mathfrak P}$ for the maximal ideal of the ring of integers of the completion~$K_{\mathfrak P}$ of~$K$ at~$\mathfrak P$. It is principal, generated by $\pi_{\mathfrak P} = 1+{\rm i}$. Then $\psi$ restricted to the group of units $\mathcal U_{\mathfrak P} =1 + \mathcal M_{\mathfrak P}$ at $\mathfrak P$ has kernel $1 + \mathcal M_{\mathfrak P}^3$. Note that the quotient group $(1 + \mathcal M_{\mathfrak P})/(1 + \mathcal M_{\mathfrak P}^3)$ is isomorphic to~$\Z/4\Z$, generated by the image of~${\rm i}$. Write $\psi = \prod_v \psi_v$, where~$v$ runs through all places of~$K$. As~$\psi$ is an id\`ele class character unramified outside $\mathfrak P$ and ${\rm i}$ is a unit everywhere, it follows from $\psi({\rm i}) = \prod_v \psi_v({\rm i}) = \psi_\infty({\rm i}) \psi_{\mathfrak P}({\rm i}) = 1$ that $\psi_{\mathfrak P}({\rm i}) = \psi_\infty({\rm i})^{-1} = -{\rm i}$, showing that $\psi_{\mathfrak P}$ on $(1 + \mathcal M_{\mathfrak P})/\big(1 + \mathcal M_{\mathfrak P}^3\big)$ has order~$4$. This in turn implies that~$\psi^2$ has conductor~$\mathfrak P^2$ and at the complex place it is algebraic of type~$3$. In particular, $\psi_{\mathfrak P}^2({\rm i}) = -1$.

To determine $L\big(\psi, s-\frac{1}{2}\big) = \prod\limits_{v \ne \mathfrak P, \infty} L\big(\psi_v, s - \frac12\big)$ explicitly, we distinguish two cases according to the residual characteristic~$p$ at~$v$.

Case (I) $p \equiv 3 \mod 4$. Then $p$ is inert in $K$, $Nv = p^2$, and we may choose $-p$ as a unifor\-mizer~$\pi_v$ at~$v$. By definition $\psi_\infty(-p) = -1$. It follows from $1 = \psi(-p) =\psi_\infty(-p) \psi_v(-p) \psi_{\mathfrak P}(-p)$ that $\psi_v(-p) = -\psi_{\mathfrak P}(-p)^{-1} = -1$ since $-p \equiv 1 \mod 4$. This shows that
\begin{gather*} L\left(\psi_v, s - \frac12\right) = \frac{1}{1 - (-1) Nv^{1/2 - s}} = \frac{1}{1 -(-p) p^{-2s}}.\end{gather*}

Case (II) $p \equiv 1 \mod 4$. Then $p= a^2 + b^2$ is a sum of two squares. It splits in $K$, $Nv = p$, and we may assume as a uniformizer $\pi_v = a+b{\rm i}$. The choice is unique by requiring $a \equiv 1 \mod 4$ and~$b$ even. Noting that $\pi_{\mathfrak P}^3 = 2(-1+{\rm i})$, we compute $a + b{\rm i} = \frac b2 \pi_{\mathfrak P}^3 + a + b \equiv 1+ b \mod \mathcal M_{\mathfrak P}^3$ so that $\psi_{\mathfrak P}(a+b{\rm i}) =\psi_{\mathfrak P}(1+ b) = (-1)^{b/2}$. Similar computation as above, using $\psi_\infty(a+b{\rm i}) = (a+b{\rm i})/\sqrt p$, yields $\psi_v(\pi_v) = (-1)^{b/2}(a - b{\rm i})/\sqrt p$. Thus
\begin{gather*} L\left(\psi_v, s - \frac12\right) = \frac{1}{1 - (-1)^{b/2}(a - b{\rm i}) p^{-s}}.\end{gather*}
Let $v'$ be the other place of $K$ with residual characteristic $p$. Then
\begin{gather*} L\left(\psi_{v'}, s - \frac12\right) = \frac{1}{1 - (-1)^{b/2}(a + b{\rm i}) p^{-s}}.\end{gather*}

Denote by $S_1$ the set of integral ideals of $\Z[{\rm i}]$ coprime to 2. Then each $\mathfrak I$ in $S_1$ is a principal ideal generated by a unique element $a+b{\rm i}$ with $a \equiv 1 \mod 4$ and $b$ even.
Conversely any $a+b{\rm i}$ of this form generates an integral ideal in~$S_1$. The discussion above shows that
\begin{gather*}\tilde \psi (\mathfrak I) = \tilde \psi ((a+b{\rm i})) = (-1)^{b/2}(a-b{\rm i})\end{gather*}
defines a Hecke Grossencharacter on $S_1$ (denoted $\big(|~|_K^{-1/2}\psi\big)'$ in notation of Section~\ref{ss:idele}) and
\begin{gather*} L(f_{32}, s) = L\(\psi, s - \frac12\) = \sum_{\mathfrak I \in S_1} \tilde \psi (\mathfrak I) (N \mathfrak I)^{-s},\end{gather*}
which in turn gives a $q$-expansion of $f_{32}$ as
\begin{align*}
 f_{32}(\tau) &= \sum_{\mathfrak I \in S_1} \tilde \psi (\mathfrak I) q^{N \mathfrak I} = \sum_{m,n\in \Z}(-1)^n(4m+1-2n{\rm i}) q^{(4m+1)^2+4n^2}\\
 &= \sum_{m,n\in \Z}(-1)^n(4m+1)q^{(4m+1)^2+4n^2}= \eta(4\tau)^2\eta(8\tau)^2.
\end{align*}

Note that the above formula gives a precise expression of the CM modular form $\eta(4\tau)^2\eta(8\tau)^2$. This can be obtained from the multiplier of $\eta(\tau)$, see~\cite{Serre} by Serre for more details.

By converse theorem, there is a weight $3$ cuspidal newform $g$ of level $16$ such that $L\big(\psi^2, s{-}1\big)$ $=L(g, s)$. The local factors of $L(g, s)$ can be easily determined from the above computations. Again, let $v$ be a nonarchimedean place of $K$ with odd residual characteristic $p$. When $p \equiv 3 \mod 4$, $\psi_v(\pi_v) = -1$ so that
\begin{gather*} L\big(\psi_v^2, s-1\big) = \frac{1}{1 - Nv^{1-s}} = \frac{1}{1 - p^2 Nv^{-s}}.\end{gather*}
 When $p \equiv 1 \mod 4$, choose $\pi_v = a + b{\rm i}$ with $a \equiv 1 \mod 4$ and $b$ even, we have $\psi_v(\pi_v) = (-1)^{b/2}(a - b{\rm i})/\sqrt p$ so that
\begin{gather*} L\big(\psi_v^2, s-1\big) = \frac{1}{1 - (a-b{\rm i})^2 Nv^{-s}}.\end{gather*}
Putting together, we have
\begin{gather*} L(g, s) = L\big(\psi^2, s-1\big) = \prod_{v \ne \mathfrak P, \infty} \frac {1}{1 - (a-b{\rm i})^2 Nv^{-s}} = \sum_{\mathfrak I \in S_1} \tilde \psi(\mathfrak I)^2 (N \mathfrak I)^{-s}\end{gather*}
in terms of the Grossencharacter $\tilde \psi^2$. This yields the following $q$-expansion of $g$:
\begin{align*}
 g(\tau) &= \sum_{\mathfrak I \in S_1} \tilde \psi(\mathfrak I)^2 q^{N\mathfrak I} = \sum_{m,n\in \Z} (4m+1 - 2n{\rm i})^2 q^{(4m+1)^2 + 4n^2} \\
 &= \sum_{m,n\in \Z}\big((4m+1)^2 - 4n^2\big)q^{(4m+1)^2 + 4n^2} =\eta(4 \tau)^6.
\end{align*}
Moreover, since $g$ has real Fourier coefficients, we also have $L\big(\psi^{-2}, s-1\big)=L(g, s)$.

\begin{Lemma}We have
\begin{gather*}
g(\tau)=\eta(4\tau)^6, \qquad L\big(\psi^2,s-1\big)=L\big(\eta(4\tau)^6,s\big), \qquad 
L\big(\eta(4\tau)^6,2\big)=\frac{1}{64}B(1/4,1/4)^2.
\end{gather*}
\end{Lemma}
\begin{proof} The first two statements have been established. It remains to prove the last assertion. The discussion in Section~\ref{section2} gives $g(2\tau)=\eta(8\tau)^6 =\eta(4\tau)^2\eta(8\tau)^2\({\frac 14} \theta_2(4\tau)^2\)$. By~\eqref{eq:31} and~\eqref{eq:theta4}, both $\eta(4\tau)^2\eta(8\tau)^2$ and $\theta_2(4\tau)^2$ can be expressed in terms of~$\l$, from which it follows that
\begin{gather*}
2\pi {\rm i} \cdot g(2\tau) {\rm d}\tau= \frac 1{32} (1-\l(4\tau))^{-3/4} \pFq{2}{1}{\frac 12&\frac 12}{&1}{\l(4\tau)} {\rm d}\lambda(4\tau).
\end{gather*}

Thus, using \eqref{eq:rn}, we have
\begin{align*}
L(g,1)&= 2\pi(-{\rm i}) \int_0^{{\rm i}\infty} g({\rm i}2t) {\rm d}({\rm i}2t) =-2\cdot 2\pi {\rm i} \int_0^{{\rm i}\infty} g({\rm i}2t) {\rm d}({\rm i}t)\\
&= -\frac{1}{16}\int_{\l(0)=1} ^{\l({\rm i}\infty) = 0} (1-\l)^{-3/4} \pFq{2}{1}{\frac 12&\frac 12}{&1}{\l}{\rm d}\l \\
&\overset{\eqref{eq:2F1-integral}}= \frac{1}{16\pi}\int_0^1 (1-\l)^{-3/4} \int_0^1 (u(1-u)(1-\l u))^{-1/2} {\rm d}u {\rm d}\l.
\end{align*}

The above calculation exhibits $L(g,1)$ as a double definite integral of the differential \begin{gather*} \frac{{\rm d}u {\rm d}\l}{(1-\l)^{3/4} (u(1-u)(1-\l u))^{1/2}} \end{gather*} on the variety $v^4=(1-\l)^3(u(1-u)(1-\l u))^2$. Thus we consider this value as a 2-period. It can be evaluated explicitly as
\begin{align*}
L(g,1)&= \frac{1}{16\pi}\int_0^1 (1-\l)^{-3/4} \int_0^1 (u(1-u)(1-\l u))^{-1/2} {\rm d}u {\rm d}\l \\
&\overset{\eqref{eq:recursive}}= \frac{1}{16} \frac{\G(1)\G\(\frac 14\)}{\G\(\frac 54\)} \pFq{3}{2}{\frac 12&\frac 12&1}{&1&\frac 54}{1}
 = \frac{1}{16} \frac{\G\(\frac 14\)}{\G\(\frac 54\)} \pFq{2}{1}{\frac 12&\frac 12}{&\frac 54}{1}\\
&\overset{\eqref{eq:gauss}}= \frac{1}{16} \frac{\G\(\frac 14\)}{\G\(\frac 54\)}\frac{\G\(\frac 54\)\G\(\frac 14\)}{\G\(\frac 34\)^2}
 = \frac{1}{16} \frac{\G\(\frac 14\)^4}{\G\(\frac 34\)^2\Gamma\(\frac 14\)^2}\overset{\eqref{eq:multiplication}}=\frac{\Gamma\(\frac 14\)^4}{32\pi \G\(\frac 12\)^2}=\frac 1{32 \pi} B(1/4,1/4)^2.
\end{align*}

The next step is to relate $L(g,1)$ and $L(g,2)$ using the functional equation~\eqref{eq:feq-idele} for $L(g,s)$. Recall that $\psi^2$ has conductor $\mathfrak P^2$, which has norm~$4$.
The $L$-function attached to $\psi^2$ satisfies the functional equation \begin{gather*}\Gamma_{\C}(s+1)L\big(\psi^2, s\big) = \epsilon\big(\psi^2, s\big) \Gamma_{\C}(2-s)L\big(\psi^{-2}, 1-s\big),\end{gather*} where $\Gamma_{\C}(s) = (2\pi)^{-s}\Gamma(s)$ and $\epsilon\big(\psi^2, s\big) = c N_{K/\Q} \big(\mathfrak d \mathfrak P^2\big)^{1/2-s} = c(16)^{1/2-s}$ for a constant $c$ which is the product of the local root number $c_{\mathfrak P}$ at the place~$\mathfrak P$ where $\psi^2$ ramifies and the root number~$c_\infty$ at the complex place~$\infty$ of~$K$. Replacing~$s$ with $s-1$ yields the functional equation for $L(g, s)$ of the usual type:
\begin{gather}\label{Lg}
\Gamma_{\C}(s) L(g, s) = \epsilon\big(\psi^2, s-1\big)\Gamma_{\C}(3-s)L(g, 3-s).
\end{gather}
Here we used the identity $L\big(\psi^{-2}, s-1\big) = L(g, s)$ observed above this lemma. Now we proceed to compute the constant $c$ using the results in Tate thesis \cite[Chapter~XV, Section~2.4]{Tate}. Recall that $\psi_{\mathfrak P}^2({\rm i}) = -1$ and $\pi_{\mathfrak P} = 1+{\rm i}$. The Gauss sum attached to $\psi_{\mathfrak P}^2$ is
\begin{gather*}\mathcal G\big(\psi_{\mathfrak P}^2\big) := \psi_{\mathfrak P}^2(1){\rm e}^{2 \pi {\rm i} \operatorname{Tr}_{K_{\mathfrak P}/\Q_2}(1/\pi_{\mathfrak P}^4)} + \psi_{\mathfrak P}^2({\rm i}){\rm e}^{2 \pi {\rm i} \operatorname{Tr}_{K_{\mathfrak P}/\Q_2}({\rm i}/\pi_{\mathfrak P}^4)} = -2 = -N_{K/\Q} \big(\mathfrak P^2\big)^{1/2},\end{gather*}
which implies $c_{\mathfrak P} = N_{K/\Q}\big(\mathfrak P^2\big)^{1/2}/\mathcal G\big(\psi_{\mathfrak P}^2\big)= -1$. Further, $c_\infty = 1/(-{\rm i})^{3-1} = -1$. Taking product, we obtain $c = 1$ and $\epsilon\big(\psi^2, s-1\big) = (16)^{3/2 - s}$. Finally plugging $s=1$ in~(\ref{Lg}), we get
\begin{gather*} \Gamma_\C(1) L(g,1)= \epsilon\big(\psi^2, 0\big) \Gamma_\C(2)L(g,2),\end{gather*}
which yields
\begin{gather*} L(g, 2) = (16)^{-1/2}2 \pi B(1/4, 1/4)^2/(32\pi) = B(1/4, 1/4)^2/64\end{gather*}
using the value $L(g,1) = \frac{1}{32 \cdot \pi}B(1/4, 1/4)^2$ computed above.
\end{proof}

Combining the above two lemmas, we have
\begin{gather*} 2 L(\psi, 1/2)^2 = 2 L(f_{32}, 1)^2 = \frac{1}{64} B(1/4, 1/4)^2 = L\big(\psi^2, 1\big),\end{gather*}
as asserted in Theorem \ref{thm:1}.

\section[A proof of the first identity in Theorem \ref{thm:2} using hypergeometric functions and properties of CM modular forms]{A proof of the first identity in Theorem \ref{thm:2} \\ using hypergeometric functions and properties\\ of CM modular forms}\label{section4}

For the degree-3 Fermat curve $F_3\colon X^3+Y^3=1$, another modular parametrization is given by
\begin{gather}\label{F3}
x=\sqrt[3]{\lambda(\tau)}, \qquad y=\sqrt[3]{1-\lambda(\tau)}, \qquad x^3+y^3=1.
\end{gather}
It defines a projection from $F_3$ to the modular curve $X_{\Gamma(2)}$ ramified only above the three cusps of $\Gamma(2)$. This makes $F_3$ a modular curve for the Fermat group $\Phi_3$, an index-$9$ subgroup of~$\Gamma(2)$, with the field of $\mathbb C$-rational functions being the degree-9 abelian extension $\mathbb C\big( \sqrt[3]{\lambda(\tau)}, \sqrt[3]{1-\lambda(\tau)}\big)$ of $\mathbb C(\l)$.

The differential $1$-form $\omega :={\rm d}x/y^2$ on $F_3$ gives rise to a weight $2$ cusp form $f$ for the Fermat group~$\Phi_3$. Substituting~(\ref{F3}) into this differential yields different expressions for $\omega$:
\begin{gather*}
\omega=\frac{{\rm d}x}{y^2}=\frac 1{y^2}\frac{{\rm d}x}{{\rm d}\tau}{\rm d}\tau=\frac 1{y^2}\frac{{\rm d}x}{{\rm d}q}\frac{{\rm d}q}{{\rm d}\tau}{\rm d}\tau= 2\pi {\rm i} \(q \frac 1{y^2}\frac{{\rm d}x}{{\rm d}q} \) {\rm d}\tau= \(\frac 1{y^2}\frac{{\rm d}x}{{\rm d}q}\){\rm d}q.
\end{gather*}

\begin{Lemma}\label{lem:diff1}Write $\omega=2\pi {\rm i} f(\tau){\rm d}\tau$. Then
\begin{gather*}
f(\tau)= q \frac 1{y^2}\frac{{\rm d}x}{{\rm d}q}=\frac{2^{1/3}}3\eta(\tau)^4.
\end{gather*}
\end{Lemma}

\begin{proof}One can prove it by using \eqref{eq:lambda-diff} and writing $(\l(1-\l))^{1/3}$ as eta quotients. Precisely,
\begin{align*}
2\pi {\rm i} f(\tau){\rm d}\tau & = \frac{{\rm d}x}{y^2}=\frac13(\l(1-\l) )^{-2/3} {\rm d}\l= \frac13 (\l(1-\l) )^{-2/3}\pi {\rm i}\l \theta_4(\tau)^4{\rm d}\tau\\
& = \frac{\pi {\rm i}}3 (\theta_2(\tau)\theta_3(\tau)\theta_4(\tau))^{4/3}{\rm d}\tau= \frac{\pi {\rm i}}3 (2\eta(\tau)^3)^{4/3}{\rm d}\tau.\tag*{\qed}
\end{align*}\renewcommand{\qed}{}
\end{proof}

\begin{Remark}\label{remark4.1} As explained at the end of Section~\ref{ss:mf}, the Fermat curve $F_3$ is in fact isomorphic to $X_0(27)$ over~$\Q$. The above lemma shows that the invariant differential of $F_3$ is expressed as the $\Q$-rational form $\eta(\tau)^4$ times an irrational multiple $2^{1/3}/3$. We know that $\eta(6\tau)^4$ is a cusp form for $\G_0(36)$. To obtain a $\Q$-rational differential, we may use the algebraic model $E_2\colon x^3+y^3=1/4$ for the Fermat curve so that it is isogenous to~$X_0(36)$ over~$\Q$.
\end{Remark}

We now compute the period $L\big(\eta(6 \tau)^4,1\big)$ using Lemma \ref{lem:diff1} and the fact $\l({\rm i}\infty)=0$, $\l(0)=1$.

\begin{Lemma}\label{lem:Lvalue-36}
\begin{gather*}L\big(\eta(6\tau)^4,1\big)=\frac1{3\cdot 2^{4/3}}B(1/3,1/3) = \frac {1}{3^{5/4}\cdot2^{5/6}}b_{\Q(\sqrt{-3})}.\end{gather*}
\end{Lemma}

\begin{proof}Using Lemma \ref{lem:diff1}, we compute the period $L\big(\eta(6 \tau)^4,1\big)$ as follows:
\begin{align*}
L\big(\eta(6\tau)^4,1\big)&=2\pi\int_0^\infty \eta({\rm i}6t)^4{\rm d}t=-{\rm i}\frac 16\cdot 2\pi\int_0^{{\rm i}\infty} \eta(\tau)^4{\rm d}\tau \\
&=-\frac{2 \pi {\rm i}}6 \frac{1}{2^{1/3}2\pi {\rm i}} \int_{1}^0 (\l(1-\l))^{-2/3}{\rm d}\l=\frac {1}{3\cdot2^{4/3}}\int_{0}^1 (\l(1-\l))^{-2/3} {\rm d}\l \\
&\overset{\eqref{eq:beta}}=\frac {1}{3\cdot2^{4/3}}B(1/3,1/3)=\frac {1}{3^{5/4}\cdot2^{5/6}}b_{\Q(\sqrt{-3})}.
\end{align*} The last equality follows from $B(1/3,1/3)=(4/3)^{1/4}b_{\Q(\sqrt{-3})}$.
\end{proof}

The elliptic curve $X_0(36)$ has CM by $\Q\big(\sqrt{-3}\big)$. Consequently there is a character $\chi$ of the id\`ele class group of~$\Q(\sqrt{-3})$, algebraic of type $2$, such that $L(\chi, s-1/2)$ is equal to the Hasse--Weil $L$-function of $X^3+Y^3=1/4$. The weight-2 normalized Hecke eigenform of level $36$, denoted~$f_{36}$, is the cusp form such that $L(f_{36}, s) = L(\chi, s - 1/2)$. Similarly, higher powers of $\chi$ also give rise to modular forms of higher weights. We shall identify some of these forms explicitly. We begin with the description of~$\chi$.

Let $\zeta_n = {\rm e}^{2 \pi {\rm i}/n}$ be a primitive $n$th root of unity. The imaginary quadratic field $F = \Q(\zeta_6)$ has class number $1$ and discriminant $-3$ so that $3$ is the only prime ramified in~$F$. Further a~prime~$p$ is inert in $F$ if and only if $p \equiv 2 \mod 3$. The ring of integers of $F$ is $\Z[\zeta_6]$. Let~$S_2$ be the set of nonzero integral ideals of $\Z[\zeta_6]$ coprime to~$6$.

\begin{Proposition}\label{prop:4.1} Each ideal $\mathfrak I \in S_2$ has a unique generator of the form $a + b \sqrt{-3}$ with $a, b \in \Z$, $a \equiv 1 \mod 3$ and exactly one of $a$, $b$ is even.
\end{Proposition}

\begin{proof} The uniqueness of such kind of generator is straightforward to check. We prove the existence. Observe that it suffices to show that $\mathfrak I$ has a generator of the form $m+ n \sqrt{-3}$ with $m, n \in \Z$. If so, then $N(\mathfrak I) = m^2 + 3 n^2$, being coprime to $3$, is congruent to $1$ modulo $3$, implying $m \equiv \pm 1 \mod 3$. Hence we have $(a, b) = (m, n)$ or $(-m, -n)$ according as $m \equiv 1$ or $-1 \mod 3$. Further, since $N(\mathfrak I)$ is odd, exactly one of~$a$,~$b$ is even.

An ideal $\mathfrak I \in S_2$ has a generator of the form $c + d \zeta_6$ with $c, d \in \Z$ and not both even. If $d = 2n$ is even, then $c + d \zeta_6 = c + n\big(1 + \sqrt{-3}\big) = (c+n) + n \sqrt{-3}$.
If $d$ is odd and $c=2m$ is even, consider the generator $\zeta_6^2(c + d \zeta_6) = -d + m\big({-}1+\sqrt{-3}\big)= (-d-m) + m \sqrt{-3}$. Finally suppose~$c$ and~$d$ are both odd. The generator $\zeta_6(c + d \zeta_6) = d \zeta_6^2 + c \zeta_6 = d(\zeta_6 - 1) + c \zeta_6 = -d + (d + c)\zeta_6$ with $c+d$ even has the desired form.
\end{proof}

\begin{Remark}\label{remark4.2} The above argument shows that any nonzero ideal $\mathfrak I$ of $\Z[\zeta_6]$ with norm coprime to $3$ is generated by an element $x= a + b \sqrt{-3}$ with $a, b \in \Z$ and $a \equiv 1 \mod 3$. If, in addition, $N(\mathfrak I)$ is even, then~$x$, $x \zeta_3$ and $x\zeta_3^2$ are the only generators of~$\mathfrak I$ of this form.
\end{Remark}

Note that if $\mathfrak I = \big(a + b \sqrt{-3}\big)$ and $\mathfrak J = \big(c + d\sqrt{-3}\big)$ are two ideals in $S_2$ with generators of the desired form, then $\big(a+b\sqrt{-3}\big)\big(c + d\sqrt{-3}\big)$ is the desired generator of $\mathfrak {I J}$.

Denote by $v_2$ (resp.~$v_3$) the only place of $F$ above $2$ (resp.~$3$); it has norm $Nv_2 = 4$ (resp.\ \smash{$Nv_3 = 3$}). For $j \in \{2, 3\}$ denote by $\mathcal M_j$ the maximal ideal at~$v_j$, and~$\mathcal U_j$ the group of units at~$v_j$. Let $\chi$ be the id\`ele class character of $F$ so that $L\big(\chi, s - \frac12\big)$ is the Hasse--Weil $L$-function of the elliptic curve $X_0(36)$ over~$\Q$. Since the conductor $36$ of the elliptic curve $X_0(36)$ is the product of the absolute value of the discriminant $-3$ and the norm of the conductor of $\chi$, we conclude that $\chi$ has conductor $\mathcal M_2 \mathcal M_3$. Denote by $\chi_v$ the restriction of $\chi$ at the place $v$ of $F$. By definition, at the complex place $\infty$ of $F$, $\chi_\infty(z) = z/|z|$. Since $\chi_{v_3}$ is nontrivial on {$\mathcal U_{3}/(1+\mathcal M_3)$}, which is generated by~$-1$, we have $\chi_{v_3}(-1) = -1$. As $\zeta_3$ is a unit of $F$ of order $3$, we have $\chi_{v_3}(\zeta_3) = 1$. The group $\mathcal U_2/(1 + \mathcal M_2)$ is cyclic of order $3$, generated by $\zeta_3$. Therefore we have $\chi_{v_2}(\zeta_3) = \chi_\infty(\zeta_3)^{-1} = \zeta_3^2$.

Parallel to what we did for $\psi$, we determine the value of $\chi_v(\pi_v)$ at a uniformizer $\pi_v$ at the place $v$ of $F$ not equal to $ v_2$, $v_3$, $\infty$ according to its residual characteristic~$p$.

Case (I). $p \equiv 2 \mod 3$ and $p \ne 2$. Then $p$ is odd and $Nv = p^2$. Choose $\pi_v = -p$ so that $\chi_{v_3}(-p) = 1$ and $\chi_{v_2}(-p)=1$. Then $\chi_v(-p) = \chi_\infty(-p)^{-1} = -1$.

Case (II). $p \equiv 1 \mod 3$. Then $p$ splits in $F$ so that $Nv = p$. By the proposition above, we may choose $\pi_v = a + b \sqrt{-3}$ with $a, b \in \Z$, $a \equiv 1 \mod 3$ and exactly one of $a$, $b$ is even. Observe that $\chi_{v_3}\big(a + b \sqrt{-3}\big) = \chi_{v_3}(1)=1$.
Next we compute $\chi_{v_2}\big(a + b \sqrt{-3}\big)$. If $a$ is even, then $\chi_{v_2}\big(a + b \sqrt{-3}\big) = \chi_{v_2}\big(\sqrt{-3}\big) = 1$ since its square is $\chi_{v_2}(-3) = \chi_{v_2}(1) = 1$ and $\chi_{v_2}$ on $\mathcal U_2$ has order $3$. If $b$ is even, then $\chi_{v_2}\big(a + b \sqrt{-3}\big) = \chi_{v_2}(1) = 1$. So we always have $\chi_{v_2}\big(a + b \sqrt{-3}\big) =1$. Together with $\chi_\infty\big(a+ b \sqrt{-3}\big) = \big(a + b \sqrt{-3}\big)/\sqrt p$, we conclude that $\chi_v(\pi_v) = \chi_v\big(a + b \sqrt{-3}\big) = \big(a - b \sqrt{-3}\big)\sqrt p$.

Then $L\big(\chi, s - \frac12\big) = \sum\limits_{\mathfrak I \in S_2} \tilde \chi(\mathfrak I)N(\mathfrak I)^{-s}$, where $\tilde \chi$, equal to $\big(|~|_F^{-1/2}\chi\big)'$ in the notation of Section~\ref{ss:idele}, is the Grossencharacter on~$S_2$ whose value at $\mathfrak I = \big(a+ b \sqrt{-3}\big)$ with $a \equiv 1 \mod 3$ is given by
\begin{gather*} \tilde \chi (\mathfrak I) = \tilde \chi\big(a + b \sqrt{-3}\big) = a - b \sqrt{-3}.\end{gather*}
Since $L(f_{36}, s) = L\big(\chi, s - \frac12\big)$, we obtain an expression for the $q$-expansion of $f_{36}$:
\begin{align}
f_{36}(\tau) & = \sum_{m, n \in \Z\atop{~m \equiv n \mod 2}} \big(3m+1 - n \sqrt{-3}\big)q^{(3m+1)^2 + 3n^2} \nonumber\\
& = \sum_{m, n \in \Z \atop{m \equiv n \mod 2}} (3m+1)q^{(3m+1)^2 + 3n^2} = \eta(6 \tau)^4.\label{f36}
\end{align}

Next we identify the modular forms corresponding to $\chi^k$, predicted by the converse theorem, for the first couple values of $k$. When $k=2$, there is a cusp form $h_3$ of weight $3$ so that $L(h_3, s) = L\big(\chi^2, s - 1\big)$. When $k=3$, there is a cusp form $h_4$ of weight $4$ so that $L(h_4, s) =L\big(\chi^3, s - \frac32\big)$. Note that~$\chi^2$ has conductor $\mathcal M_2$ so that $h_3$ has level $12$ while $\chi^3$ has conductor $\mathcal M_3$ so that $h_4$ has level $9$. For each case the corresponding spaces of cusp forms are 1-dimensional from which one can conclude $h_3=\eta(2\tau)^3\eta(6\tau)^3$ and $h_4=\eta(3\tau)^8$ right away. Here we give a bit more details. To write the corresponding $L$-functions, we choose $\pi_{v_3} = \sqrt{-3}$ and $\pi_{v_2} = -2$ and compute $\chi_{v_3}(\pi_{v_3})= \chi_\infty(\sqrt{-3})^{-1} \chi_{v_2}(\sqrt{-3})^{-1} = {\rm i}^{-1} = -{\rm i}$ and $\chi_{v_2}(\pi_{v_2}) = \chi_\infty( -2)^{-1} \chi_{v_3}( -2)^{-1}=-1$. Thus $\chi_{v_3}^2(\pi_{v_3}) = -1$. Therefore
\begin{align*}
L(h_3, s) &= L\big(\chi^2, s - 1\big) = \frac{1}{1 -(-3) 3^{-s}} \cdot \sum_{\mathfrak I \in S_2} \tilde \chi^2(\mathfrak I)N(\mathfrak I)^{-s}\\
&=\sum_{k \ge 0}(-3)^k 3^{-ks} \cdot \sum_{m, n \in \Z \atop{m \equiv n \mod 2}} \big(3m+1 - n \sqrt{-3}\big)^2\big((3m+1)^2 + 3n^2\big)^{-s}\\
&=\sum_{k\ge 0, m, n \in \Z\atop{m \equiv n \mod 2}} (-3)^k\big((3m+1)^2 - 3n^2\big)\big(3^k\big((3m+1)^2 + 3n^2\big)\big)^{-s}.
\end{align*}
This in turn gives a $q$-expansion for $h_3$ as follows:
\begin{gather*}
h_3(\tau)= \eta(2\tau)^3\eta(6\tau)^3 = \sum_{k \ge 0, m, n \in \Z\atop{m \equiv n \mod 2}}(-3)^k \big((3m+1)^2 - 3n^2\big) q^{3^k((3m+1)^2 + 3n^2)}.
\end{gather*}

Similarly, $\chi_{v_2}^3(\pi_{v_2})= -1$ and
\begin{gather*}
\begin{split}
& L(h_4, s) = L\left(\chi^3, s - \frac32\right) = \frac{1}{1 - (-8)4^{-s}}\cdot \sum_{\mathfrak I \in S_2} \tilde \chi^3(\mathfrak I)N(\mathfrak I)^{-s}\\
& \hphantom{L(h_4, s)}{} = \sum_{k \ge 0}(-8)^k4^{-ks} \cdot \sum_{m, n \in \Z\atop{m \equiv n \mod 2}} \big(3m+1 - n \sqrt{-3}\big)^3\big((3m+1)^2 + 3n^2\big)^{-s}\\
& \hphantom{L(h_4, s)}{}= \sum_{k \ge 0, ~m, n \in \Z\atop{m \equiv n \mod 2}} (-8)^k\big((3m+1)^3 - (3m+1)9n^2\big)\big(4^k\big((3m+1)^2 + 3n^2\big)\big)^{-s},
\end{split}
\end{gather*} which in turn yields a $q$-expansion of $h_4$ as
\begin{gather*}
h_4(\tau) = \eta(3 \tau)^8= \sum_{k \ge 0, ~ m, n \in \Z\atop{m \equiv n \mod 2}} (-8)^k\big((3m+1)^3 - (3m+1)9n^2\big)q^{4^k((3m+1)^2 + 3n^2)}.
\end{gather*}The reader is referred to~\cite{Serre} by Serre for discussions of $q$-expansions of powers of $\eta(\tau)$ when they satisfy complex multiplication.

Using the identification $f_{36}(\tau) = \eta(6 \tau)^4$, we restate Lemma~\ref{lem:Lvalue-36} as
\begin{gather*}L(f_{36},1)=\frac1{3\cdot 2^{4/3}}B(1/3,1/3) = \frac {1}{3^{5/4}\cdot2^{5/6}}b_{\Q(\sqrt{-3})}.\end{gather*}

We proceed to compute the period $L(h_3,2)$. First, using notation in Section~\ref{ss:mf} and setting $s(\tau) = c(\tau)/a(\tau)$, we rewrite $h_3(\tau/2)$ as
\begin{align*}
h_3(\tau/2)&={{\eta(\tau)^3\eta(3\tau)^3}} =\(\frac{bc}3\)^{3/2}=3^{-3/2}\big(s\big(1-s^3\big)^{1/3}a^2\big)^{3/2}\\
 &=\big(s\big(1-s^3\big)a^2/3\big)\big(s^{1/2}\big(1-s^3\big)^{-1/2}3^{-1/2} a\big).
\end{align*} Using \eqref{eq:ds}, one has
\begin{align*}
L(h_3(\tau/2),1)& = -2 \pi {\rm i} \int_0^{{\rm i}\infty}h_3(\tau/2){\rm d}\tau \overset{\eqref{eq:a-2F1}}= \frac{1}{\sqrt{3}}\int_0^1 s^{1/2}\big(1-s^3\big)^{-1/2} \pFq{2}{1}{\frac 13&\frac 23}{&1}{s^3} {\rm d}s\\
& \overset{u:=s^3}= \frac{1}{3\sqrt 3}\int_0^1 u^{-1/2}(1-u)^{-1/2} \pFq{2}{1}{\frac 13&\frac 23}{&1}{u} {\rm d}u \\
& \overset{\eqref{eq:recursive}}= \frac{\pi}{3\sqrt 3} \pFq{3}{2}{\frac 13&\frac 23&\frac 12}{&1&1}{1} .
\end{align*}
This expression leads to the period $L(h_3,2)$ by applying \eqref{eq:rn}:
\begin{align*}
L(h_3,2)=&(2\pi)^2 \int_0^{\infty}h_3({\rm i}t)t{\rm d}t= (2\pi)^2 \frac14\int_0^{\infty}h_3({\rm i}t/2)t{\rm d}t\\
=&- \pi^2 \int_0^{{\rm i}\infty}h_3(\tau/2)\tau {\rm d}\tau =- \frac{\pi^2}{9} \int_0^{{\rm i}\infty}\frac{1}{\tau^3} \eta\(\frac{-1}{3\tau}\)^3\eta\(\frac{-1}{\tau}\)^3 {\rm d}\tau\\
=&- {\rm i} \frac{\pi^2}{9} 3^{3/2} \int_0^{{\rm i}\infty} \eta\({3\tau}\)^3\eta\({\tau}\)^3 {\rm d}\tau = \frac{\pi}{2\sqrt 3} L(h_3(\tau/2),1)
= \frac{\pi^2}{18} \pFq{3}{2}{\frac 13&\frac 23&\frac 12}{&1&1}{1}.
\end{align*}

In the above computation, the 4th equality results from the change of variable replacing $\tau$ by $-1/3\tau$, and the 5th equality uses the transformation identity for $\eta$ function
\begin{gather*} \eta\(-\frac{1}{\tau}\) = \sqrt{-{\rm i}\tau}\eta(\tau).\end{gather*}
Note that the 4th and the 5th steps combined relates $L(h_3, 2)$ to $L(h_3, 1)$, which could be computed using the functional equation for $L\big(\chi^2, s\big)$ or $L(h_3,s)$ as we did in the previous section. Here we took advantage of the fact that $h_3(\tau/2)$ is an $\eta$-product to compute directly the relation between $h_3(\tau)$ and $h_3(-1/(12\tau))$, which is the inverse Mellin transform of the functional equation relating $L(h_3,s)$ and $L(h_3, 3-s)$.

To finish, note the relation between special values
\begin{align*}
\pFq{3}{2}{\frac 13&\frac 23&\frac 12}{&1&1}{1} & =\pFq{2}{1}{\frac16&\frac 13}{&1}{1}^2=\frac{\G(1/2)^2\G(1)^2}{\G(2/3)^2\G(5/6)^2}\\
& \overset{\eqref{eq:multiplication}}=\frac{1}{4\pi^2}\frac{\G(1/2)^2\G(1/6)^2}{\G(2/3)^2}=\frac{3\cdot 2^{1/3}}{8\pi^2} B(1/3,1/3)^2.
\end{align*}The first equality above results from the Clausen formula \eqref{eq:Clausen},
and the Gauss summation formula \eqref{eq:gauss} is used to prove the second equality. Therefore,
\begin{gather*}
L(h_3, 2)= L\big(\eta(2\tau)^3\eta(6\tau)^3,2\big)=\frac{2^{1/3}}{48} B(1/3,1/3)^2=\frac{3}{2}L(f_{36},1)^2.
\end{gather*}
As observed before, $L(f_{36}, 1) = L(\chi, 1/2)$ and $L(h_3, 2) = L\big(\chi^2, 1\big)$, this proves the first identity in Theorem~\ref{thm:2} for both id\`ele class characters and cusp forms.

\section[Computing $L$-values using Eisenstein series and a proof of the second identity in Theorem \ref{thm:2}]{Computing $\boldsymbol{L}$-values using Eisenstein series\\ and a proof of the second identity in Theorem \ref{thm:2}}
Now we switch our approach to use CM values of Eisenstein series discussed in Section~\ref{ss:2.5}. This will allow us to compute more values of $L(\chi^k,k/2)$ systematically. Denote by~$\Sigma$ the set of nonzero ideals of $\Z[\zeta_6]$ coprime to $\mathcal M_3$. Then $\Sigma = \cup_{i \ge 0} (-2)^iS_2 = S_2 \cup (-2)\Sigma$.

Recall from the discussion in Section~\ref{section4} that except for the local factors at $v_2$ and $v_3$, $L\big(\chi^k, s-k/2\big)$ agrees with $\sum_{\mathfrak I \in S_2} \tilde \chi^k(\mathfrak I)N(\mathfrak I)^{-s}$. Furthermore, the conductor of $\chi^k$ is
\begin{gather*} \begin{cases}\mathcal M_2\mathcal M_3& \text{if } k\equiv 1, 5 \mod 6,\\
\mathcal M_2 & \text{if } k\equiv 2,4 \mod 6,
\\
\mathcal M_3 & \text{if } k\equiv 3 \mod 6,
\\
{\mathcal O}_F & \text{if } k\equiv 0 \mod 6,
\end{cases}\end{gather*}respectively.
Here $\mathcal O_F = \Z[\zeta_6]$ is the ring of integers of $F$. It follows from definition \eqref{eq:L-idele} that
\begin{gather}\label{eq:chi-N}L\big(\chi^k,s-k/2\big)=\sum_{\mathfrak I \in S_2} \tilde \chi^k(\mathfrak I)N(\mathfrak I)^{-s} \cdot \begin{cases}1
& \text{if } k\equiv \pm 1\mod 6,\\
\dfrac 1{1-(-3)^{k/2}3^{-s}} &\text{if } k\equiv \pm 2\mod 6,\\
{{\dfrac 1{1-(-2)^{k}4^{-s}}}}&\text{if } k\equiv 3\mod 6,\\
{{\dfrac 1{1-(-2)^{k}4^{-s}}}}\frac 1{1-(-3)^{k/2}3^{-s}} &\text{if } k\equiv 0\mod 6.
\end{cases}\hspace*{-10mm}
\end{gather}

Observe that, for $k \equiv 0 \mod 3$, $\chi^k$ is unramified at $v_2$, and we can extend the definition of~$\tilde \chi^k$ on $S_2$ to $\Sigma$ by the same formula, that is, for $\mathfrak I = \big(a + b \sqrt{-3}\big) \in \Sigma$ with $a, b \in \Z$ and $a \equiv 1 \mod 3$, let
\begin{gather*} \tilde \chi^k (\mathfrak I) = \big(a - b \sqrt{-3}\big)^k.\end{gather*}
This is well-defined also for $\mathfrak I$ with even norm because in this case, by Remark~\ref{remark4.1}, the three generators of the above form differ by a power of $\zeta_3$, and the difference disappears after raising to the $k$th power. In this case, noting $\Sigma = \cup_{i \ge 0} (-2)^iS_2$ and using the multiplicativity of $\tilde \chi^k$ on~$\Sigma$, we have
\begin{gather}\label{eq:r=01}
\sum_{\mathfrak I \in \Sigma} \tilde \chi^k(\mathfrak I)N(\mathfrak I)^{-s}=\frac 1{1-(-2)^k 4^{-s}}\sum_{\mathfrak I \in S_2} \tilde \chi^k(\mathfrak I)N(\mathfrak I)^{-s} = L\big(\chi^k, s - k/2\big).
\end{gather}

To continue, note that the Fourier expansion of $f_{36}$ given in {(\ref{f36})} can be extended to
\begin{gather*}
f_{36}(\tau)=\eta(6\tau)^4=\sum_{m, n \in \Z} \big(3m+1 - n \sqrt{-3}\big)q^{(3m+1)^2 + 3n^2} = \sum_{m, n \in \Z} (3m+1)q^{(3m+1)^2 + 3n^2} .
\end{gather*}
The extra terms come from $m$ and $n$ having opposite parity. In this case $x= 3m+1 - n \sqrt{-3}$ generates an ideal in $\Sigma$ with even norm $(3m+1)^2 + 3n^2$, and by Remark~\ref{remark4.2}, $x \zeta_3$ and $x \zeta_3^2$ are the two other elements of this form and with the same norm. Since these three elements sum to zero, no additional contributions result.

Now consider the following generalization to $k \ge 1$:
\begin{gather*}\sum_{\mathfrak J\in S_2} \tilde \chi^k(\mathfrak J) q^{N(\mathfrak J)} = \sum_{m,n\in \Z\atop{m \equiv n \mod 2}} \big(3m+1-n\sqrt{-3}\big)^k q^{(3m+1)^2+3n^2},\end{gather*}
which is a partial sum of
\begin{gather*}G_k(q):= \sum_{m,n\in \Z} \big(3m+1- n\sqrt{-3}\big)^k q^{(3m+1)^2+3n^2},\end{gather*}
whose associated Dirichlet $L$-series is \begin{gather*}L(G_k(q),s)=\sum_{m,n\in \Z} \big(3m+1-n\sqrt{-3}\big)^k \big((3m+1)^2+3n^2\big)^{-s}.\end{gather*}
Comparing these two $q$-series, we note that, just like the argument for $f_{36}$ above, each extra term in~$G_k(q)$ from $m, n$ with opposite parity will have $x = 3m+1 - n \sqrt{-3}$ generate an ideal in~$\Sigma$ with even norm $(3m+1)^2 + 3n^2$, and $x\zeta_3$ and $x \zeta_3^2$ are precisely the other two elements of the same form and with the same norm. Together they contribute $x^k + (x\zeta_3)^k + \big(x\zeta_3^2\big)^k$ to the coefficient of $q^{(3m+1)^2+3n^2}$ in~$G_k(q)$, which is zero if $k$ is not a multiple of $3$, and $3\big(3m+1-n\sqrt{-3}\big)^k$ otherwise. In conclusion, we have shown
\begin{itemize}\itemsep=0pt
\item When $k$ is not a multiple of $3$,
\begin{gather}
G_k(q)=\label{eq:r<>0}\sum_{\mathfrak J\in S_2} \tilde \chi^k(\mathfrak J) q^{N(\mathfrak J)}.\end{gather}

\item When $k\equiv 0\mod 3$, using the extended $\tilde \chi^k$ on $\Sigma = S_2 \cup (-2)\Sigma$, we can write
\begin{align}G_k(q)& =\sum_{\mathfrak J\in S_2} \tilde \chi^k(\mathfrak J) q^{N(\mathfrak J)}+3 \sum_{\mathfrak J\in \Sigma} \tilde \chi^k((-2)\mathfrak J) q^{N((-2)\mathfrak J)}\nonumber\\
& =\sum_{\mathfrak J\in S_2} \tilde \chi^k(\mathfrak J) q^{N(\mathfrak J)}+3(-2)^k \sum_{\mathfrak J\in \Sigma} \tilde \chi^k(\mathfrak J) q^{4N(\mathfrak J)}.\label{eq:r=0}
\end{align}
\end{itemize}

Combining \eqref{eq:r=01}, \eqref{eq:r<>0}, \eqref{eq:r=0}, and \eqref{eq:chi-N} leads to the following relation between $L\big(\chi^k,k/2\big)$ and $L( G_k(q),k)$:
\begin{gather}\label{eq:chi-L}
L\big(\chi^k, k/2\big)=L( G_k(q),k)\cdot \begin{cases}1&\text{if } k\equiv \pm 1 \mod 6,\\
\dfrac1{1-(-3)^{k/2}3^{-k}}&\text{if } k\equiv \pm 2 \mod 6,\\
\dfrac1{1+2(-2)^k4^{-k}} &\text{if } k\equiv 3 \mod 6,\\ \dfrac1{(1-(-3)^{k/2}3^{-k})(1+2(-2)^k4^{-k})} &\text{if } k\equiv 0 \mod 6.
\end{cases}
\end{gather}
Thus for $k=2$, the constant $\frac 1{1-(-3)^{k/2}3^{-k}}$ is $\frac 34$; when $k=3$, the factor $(1+2(-2)^k4^{-k})^{-1}=\frac43$.

\begin{Remark} As noted in Section~\ref{ss:idele}, corresponding to each $\chi^k$ with $k \ge 1$ there is a modular form $h_{k+1} = \sum\limits_{n \ge 0} a_{k+1}(n)q^n$ of weight $k+1$ such that $L(h_{k+1}, s) = L\big(\chi^k, s - \frac{k}{2}\big)$. The function~$G_k(q)$ defined above is obtained from $h_{k+1}$ by removing the Fourier coefficients $a_{k+1}(n)$ with~$n$ not coprime to~$6$. So it is also a modular form of weight $k+1$. When $k \equiv \pm 1 \mod 6$, it is equal to $h_{k+1}$, and for other~$k$ its level is higher than that of $h_{k+1}$.
\end{Remark}

The relations between $L\big(\chi^k,k/2\big)$ and $L(G_k(q),k)$ described in \eqref{eq:chi-L} are useful for determining the exact values of $L\big(\chi^k,k/2\big)$ since $L(G_k(q),k)$ can be computed using CM values of explicit Eisenstein series.

The Mellin transform of $f_{36}({\rm i}t)$, defined by
\begin{gather*}
(2\pi)^{-s}\G(s)L(f_{36},s)=\int_0^\infty f_{36}({\rm i}t) t^s \frac{{\rm d}t}{t}
\end{gather*}for $\Re(s) > 3/2$, can be analytically extended to the whole $s$-plane. Using this definition we computed the value $L(f_{36}, 1)= L\big(\eta(6 \tau)^4, 1\big)$ in Lemma~\ref{lem:Lvalue-36}. On the other hand, for $\Re (s)>3/2$, $L(f_{36},s)$ agrees with the following absolutely convergent Dirichlet $L$-series
\begin{gather*} L(f_{36}, s) = \sum_{m,n\in \Z} \frac{(3m+1)-n\sqrt{-3}}{\big((3m+1)^2+3n^2\big)^s} = \sum_{m,n\in \Z} \!\frac{1}{\big((3m+1)\! + n \sqrt{-3}\big)\big|(3m+1)\! + n \sqrt{-3}\big|^{2s-2}}.
\end{gather*}

Recall the Eisenstein series
\begin{gather*}
\mathcal G_{1,(1:3)}^\ast(s,\tau)=\Im(\tau)^s\sum_{m,n\in\Z}\frac{1}{\((3m+1)\tau+n\)}\frac{1}{|(3m+1)\tau+n|^{2s}}
\end{gather*}
defined in Section~\ref{ss:2.5}, which has analytic continuation to $s=0$ for a fixed $\tau$. By taking $\tau$ to be $\frac{-1}{\sqrt{-3}} ={\rm i}\sqrt 3 $, we find, for $\Re (s) > 2$,
\begin{gather*}
\mathcal G_{1,(1:3)}^\ast\(s,\frac{-1}{\sqrt{-3}}\) = \(\frac{-1}{\sqrt{-3}}\)^{-1-s}\sum_{m,n\in \Z}\frac{1}{\big((3m+1) + n \sqrt{-3}\big)\big|(3m+1) + n \sqrt{-3}\big|^{2s}}\\
\hphantom{\mathcal G_{1,(1:3)}^\ast\(s,\frac{-1}{\sqrt{-3}}\)}{} = \(\frac{-1}{\sqrt{-3}}\)^{-1-s}L(f_{36}, s+1).
\end{gather*} This gives another expression of the analytic continuation of $L(f_{36}, s)$ and, by letting $s \to 0^+$, we obtain the relation
\begin{gather}\label{eq:43}
L(f_{36}, 1) = \(\frac{-1}{\sqrt{-3}}\) \mathcal G_{1,(1;3)}^\ast \big({-}1/\sqrt{-3}\big),
\end{gather} which gives another way to compute the value $L(f_{36}, 1)$.

When $k\ge 2$, we have, for $\Re(s)>k$,
\begin{align*}
L(G_k(q),s)&=\sum_{m,n\in \Z}\frac{\big((3m+1)-n\sqrt{-3}\big)^k}{\big((3m+1)^2+3n^2\big)^s} \\
& = \sum_{m,n\in \Z}\frac{1}{\big((3m+1) + n \sqrt{-3}\big)^k\big|(3m+1) + n \sqrt{-3}\big|^{2s-2k}}\\
&= \(\frac{-1}{\sqrt{-3}}\)^{k}\(\frac{1}{\sqrt{3}}\)^{2s-2k}\sum_{m,n\in \Z}\frac1{\big((3m+1)\frac{-1}{\sqrt{-3}}+ n\big)^k\big|(3m+1)\frac{-1}{\sqrt{-3}} + n\big|^{2s-2k}}.
\end{align*}
This leads to a uniform expression for the $L$-value of $G_k$ at $k$ for $k \ge 1$:
\begin{gather*}L(G_k(q),k)= \(\frac{-1}{\sqrt{-3}}\)^k \mathcal G_{k,(1;3)}^\ast \big({-}1/\sqrt{-3}\big),\end{gather*} where
\begin{gather*}
\mathcal G_{k,(1;3)}^\ast(\tau):=\lim_{\Re(s)>0 \atop s\rightarrow 0} \sum_{(m,n)\in \Z}\frac1{((3m+1)\tau+ n)^k}\frac{\Im(\tau)^s}{|(3m+1)\tau+ n|^{2s}}
\end{gather*} is defined in Section~\ref{ss:2.5}. Notice that when $k\ge 3$, $\mathcal G_{k,(1;3)}^\ast(\tau) = \mathcal G_{k,(1;3)}(\tau)$ is a holomorphic Eisenstein series, and the double series converges absolutely so that there is no need to take limit to get the value $L(G_k(q),k)$.

The Fourier expansions of $\mathcal G_{k,(1;3)}^\ast(\tau)$ below are derived from Theorem~\ref{thm:G*} in Section~\ref{ss:2.5}.
\begin{Proposition}\label{prop:5.1}
\begin{gather*}
\mathcal G_{1,(1;3)}^\ast(\tau)=- \frac {{\rm i}\pi}{3}-2\pi {\rm i}\sum_{n\ge 1}\frac{q^{n}-q^{2n}}{1-q^{3n}},\\
\mathcal G_{2,(1;3)}^\ast(\tau)=- \frac {\pi}{3\Im(\tau)}-(2\pi)^2\sum_{n\ge 1}n\frac{q^{n}+q^{2n}}{1-q^{3n}},
\end{gather*}
and for $k\ge3$,
\begin{gather*}
\mathcal G_{k,(1;3)}^\ast(\tau)= \frac{(-2\pi {\rm i})^k}{(k-1)!}\sum_{n\ge 1} n^{k-1}\left( \frac{q^{n}+(-1)^kq^{2n}}{1-q^{3n}}\right).
\end{gather*}
\end{Proposition}

From \cite[Table~4]{Sebbar} by Sebbar, one knows that the graded algebra of holomorphic integral weight modular forms for $\G(3)$ is a polynomial algebra generated by \begin{gather*}\phi(\tau):=\theta_2(2\tau)\theta_2(6\tau)+\theta_3(2\tau)\theta_3(6\tau), \qquad \text{and} \qquad \phi_1(\tau):=\frac 16(\phi(\tau/3)-\phi(\tau)).\end{gather*} By comparing Fourier expansions, we conclude from the above Proposition that ${{\mathcal G_{1,(1;3)}^\ast}}=-2\pi {\rm i}\phi/6$, ${{\mathcal G_{3,(1;3)}^\ast}}=-\frac{(2\pi {\rm i})^3}{2}\phi_1^3$, and ${{\mathcal G_{4,(1;3)}^\ast}}= \frac{(2\pi {\rm i})^4}{6}\phi\phi_1^3$. Note that $\mathcal G_{2,(1;3)}^\ast(\tau)$ is not holomorphic, and we have
\begin{gather}\label{eqn:E2G2}
\mathcal G_{2,(1;3)}^\ast(\tau)=\frac{\pi^2}{3}\(E_2^\ast(3\tau)-\phi^2(\tau)\),
\end{gather}
where
\begin{gather*}
E_2^\ast(\tau)=1-\frac{3}{\pi \Im(\tau)}-24\sum_{n> 0}n\frac{q^{n}}{1-q^{n}}
\end{gather*}
is the non-holomorphic weight-$2$ Eisenstein series. See~\cite{Zagier} for more details.

\begin{Theorem} \label{thm: phi values}Let $\tau_0=\frac{\rm i}{\sqrt{3}}$. Then
\begin{gather*}
\phi(\tau_0)= \frac{\sqrt{3}}{2\pi}\frac1{2^{1/3}}B(1/3,1/3) \qquad \text{and} \qquad
\phi_1(\tau_0)= \frac1{2^{2/3}\cdot 3^{1/2}}\frac{1}{2\pi}B(1/3,1/3).\end{gather*}
\end{Theorem}

Before proving this theorem, we use it to establish the second identity in Theorem~\ref{thm:2}, relating $L(h_4, 3)$ to the cube of $L(f_{36}, 1)$.
\begin{Corollary}
\begin{gather*}
L(f_{36},1)=-2\pi {\rm i} \frac{\rm i}{\sqrt 3}\frac{\phi(\tau_0)}{6}=\frac1{2^{4/3}\cdot 3}B(1/3,1/3),
\end{gather*}
and
\begin{gather*}
L(h_4, 3)= L\big(\eta(3\tau)^8,3\big)=\frac 83 L(f_{36},1)^3.
\end{gather*}
\end{Corollary}

\begin{proof}As remarked above,
\begin{gather*}
\mathcal G_{1,(1;3)}^\ast(\tau)=-2\pi {\rm i}\phi(\tau)/6.
\end{gather*}
Thus by \eqref{eq:43}
\begin{gather*}
L(f_{36},1)=-2\pi {\rm i} \frac{\rm i}{\sqrt 3}\frac{\phi(\tau_0)}{6}=\frac1{2^{4/3}\cdot 3}B(1/3,1/3),
\end{gather*}
which agrees with Lemma \ref{lem:Lvalue-36}. Further,
\begin{align*}
L\big(\eta(3\tau)^8,3\big)& =L\big(\chi^3,3/2\big)= \frac 43 L(G_3(q), 3) = \frac 43 \frac{-{\rm i}}{3^{3/2}}\mathcal G_{3,(1;3)}^\ast (\tau_0) \\
& = \frac 43 \frac12\(\frac{2\pi}{\sqrt{3}}\)^3\phi_1(\tau_0)^3 = \frac 83 L(f_{36},1)^3.
\end{align*}
This proves the corollary.
\end{proof}

The calculation before Proposition \ref{prop:5.1} shows that
\begin{gather*}
- \frac 14 \mathcal G_{2,(1;3)}^\ast(\tau_0) =\frac 34L(G_2(q), 2) = L\big(\eta(2\tau)^3\eta(6\tau)^3,2\big)=\frac 32 L(f_{36},1)^2,
\end{gather*}
where the last equality was proved in the previous section. Combining Theorems \ref{thm:2} and \ref{thm: phi values} with the identity~\eqref{eqn:E2G2}, we obtain, as a consequence, the values
\begin{gather*}
E_2^\ast\big(\sqrt{-3}\big)=2^{-8/3}\frac{B(1/3,1/3)^2}{\pi^2},\qquad \text{and} \qquad E_2\big(\sqrt{-3}\big)=\frac{\sqrt{3}}{\pi}+2^{-8/3}\frac{B(1/3,1/3)^2}{\pi^2}.
\end{gather*}

Now we turn to the proof of Theorem \ref{thm: phi values}, for which we use the expression
\begin{gather*}
\phi(\tau)=4\frac{\eta(4\tau)^2\eta(12\tau)^2}{\eta(2\tau)\eta(6\tau)}+\frac{\eta(2\tau)^5\eta(6\tau)^5}{\eta(\tau)^2\eta(4\tau)^2\eta(3\tau)^2\eta(12\tau)^2},
\end{gather*}
and the fact
\begin{gather*}
\phi(\tau)=\(\frac{\eta(\tau/3)^3}{\eta(3\tau)^3}+3\)\phi_1(\tau).
\end{gather*} Hence, the values of $\phi(\tau_0)$ and $\phi_1(\tau_0)$ can be determined by the following $\eta$-values and the value of
\begin{gather*}t :=\phi(\tau)/\phi_1(\tau)-3=\eta(\tau/3)^3/\eta(3\tau)^3.\end{gather*}

\begin{Proposition}\label{prop: eta values}
Let $\tau_0=\frac{\rm i}{\sqrt{3}}$. Then we have
\begin{gather*}
t\(\tau_0\)=3\big(2^{1/3}-1\big),
\end{gather*}
and the $\eta$-values
\begin{gather*}
\eta(\tau_0)=\frac{{3^{1/4}}}{2^{1/3}}\frac{3^{1/8}}{2\pi}\G(1/3)^{3/2}=\frac{3^{1/4}B(1/3,1/3)^{1/2}}{2^{1/3}\sqrt{2\pi}},\\
\eta(2\tau_0)=\frac{\big(104+60\sqrt3\big)^{1/24}}{\sqrt 2}\eta(\tau_0),\\
\eta(3\tau_0)=3^{-1/4}\eta(\tau_0),\\
\eta(4\tau_0)=\frac{\big(1980\sqrt 2-4544+1980\sqrt 6-1440\sqrt3\big)^{1/24}}{2}\eta(\tau_0),\\
\eta(6\tau_0)=\frac{\big(3\sqrt3-5\big)^{1/12}}{2^{11/24}}\eta(3\tau_0),\\
\eta(12\tau_0)= \frac{\big(495\sqrt6-495\sqrt2+360\sqrt3-1136\big)^{1/24}}{2^{11/12}}\eta(3\tau_0).
\end{gather*}
\end{Proposition}

\begin{proof}Let \begin{gather*} u(\tau)=(\eta(2\tau)/\eta(\tau))^{24}\end{gather*} be a Hauptmodul of $\G_0(2)$. {Up to equivalence, the two cusps of $X_0(2)$ are $0$, ${\rm i}\infty$ and the only elliptic point of $X_0(2)$ is $\frac{1+{\rm i}}2$ of order $2$.} The function $u(\tau)$ has a zero of order 1 at the cusp~${\rm i}\infty$ and a pole of order 2 at the cusp~$0$. The ramification data of the covering map from the modular curve~$X_0(2)$ to $X_0(1)$ tells us that {the relation between the modular $j$-function and $u$} must be of the form
\begin{gather*}
j=\frac{(u+a)^3}{Au} \quad \text{for some constants} \ \ a, \ A,
\end{gather*}
since the elliptic point ${\zeta_3}=\big({-}1+\sqrt{-3}\big)/2$ of $\Gamma_0(1)$ is a ramification point of order $3$ of the covering map. In addition, the elliptic point~${\rm i}$ of $\Gamma_0(1)$ is a ramification point of order~$2$.
Using the known values
\begin{gather*}
\eta({\rm i})=\frac{\G(1/4)}{2\pi^{3/4}}, \qquad \eta(2{\rm i})=\frac1{\sqrt 2}\eta({\rm i}/2)=\frac{\G(1/4)}{2^{11/8}\pi^{3/4}},
\end{gather*}
we obtain that $u({\rm i})=1/512$. (For these $\eta$-values, see \cite[Theorem~10.5.13-Examples]{Cohen-ant}.) On the other hand, the ramification data gives that $u((1+{\rm i})/2)=-1/64$ and $u({\zeta_3})=-1/256$. Thus
the relation between the elliptic $j$-function and $u$ is
\begin{gather}\label{eq:j-u}
j= \frac{(1+256u)^3}{u}.
\end{gather}

To evaluate the $\eta$-values, we will first consider the corresponding $j$-values, $u$-values, and also use the functional equation $\eta(-1/\tau)=\sqrt{\tau/{\rm i}}\eta(\tau)$. For example, for the values $\eta\big(\sqrt{-3}\big)$ and $\eta\big({\rm i}/\sqrt{3}\big)$, we use the known value
\begin{gather*}
\eta\(\frac{-1+\sqrt{-3}}2\)={\rm e}^{-\pi {\rm i}/24}\frac{3^{1/8}}{2\pi}\G(1/3)^{3/2},
\end{gather*}
and $u({\zeta_3})=-1/256$. Then we have
\begin{gather*}
\frac{\eta\big(\sqrt{-3}\big)^{24}}{\eta({\zeta_3})^{24}}=\frac{\eta(2{\zeta_3})^{24}}{\eta({\zeta_3})^{24}}= -\frac 1{256}.
\end{gather*}
We note that $\eta({\rm i}y)\in \mathbb R^+$ for $y\in \mathbb R^+$. Thus we have
\begin{gather*}
\eta\big(\sqrt{-3}\big)=\frac1{2^{1/3}}\frac{3^{1/8}}{2\pi}\G(1/3)^{3/2} \qquad \text{and} \qquad \eta\big({\rm i}/\sqrt{3}\big)=3^{1/4}\eta\big(\sqrt{-3}\big)=\frac{3^{1/4}}{2^{1/3}}\frac{3^{1/8}}{2\pi}\G(1/3)^{3/2}.
\end{gather*}

We will get the other $\eta$-values in a likewise manner using modular polynomials. Recall that the level 2 modular polynomial
\begin{gather*}
\Phi_2(X,Y)= X^3+Y^3-X^2Y^2+1488 \big(XY^2+X^2Y\big)-162000\big(X^2+Y^2\big)\\
\hphantom{\Phi_2(X,Y)=}{} +40773375XY+8748000000(X+Y)-157464000000000
\end{gather*}
is satisfied by $j(\tau)$ and $j(2\tau)$, see \cite[Section~6.1]{Zagier}. As $j({\zeta_3})=0$, we see that
\begin{gather*}
\Phi_2(X,0)=(X-54000)^3
\end{gather*}
and hence $j\big(\sqrt{-3}\big)=54000$. This value and the polynomial $\Phi_2\big(X,j\big(\sqrt{-3}\big)\big)$ allow us to determine the algebraic values $j\big(\sqrt{-3}/2\big)$ and $j\big(2\sqrt{-3}\big)$. For instance, we have
\begin{gather*}
\Phi_2(X,54000)=X\big(X^2-2835810000X+6549518250000\big).
\end{gather*}
To nail down which of the three roots is $j\big(2\sqrt{-3}\big)$, a convenient way is plugging $\tau=2\sqrt{-3}$ into the Fourier expansion of $j(\tau)$ to get an estimate of $j\big(2\sqrt{-3}\big)$. Similarly one can compute $j\big(\sqrt{-3}/4\big)$ using $\Phi_2\big(X,j\big(\sqrt{-3}/2\big)\big)$. We tabulate the $j$-values below.
\begin{gather*}\rowcolors{1}{white}{pink}
\arraycolsep=4.7pt\def\arraystretch{1.2}
\begin{array}{c|c}
\tau& j(\tau) \\ \hline\hline
\sqrt{-3}& 54000\\\hline
2\sqrt{-3}& \begin{array}{c}
1417905000+818626500\sqrt3
\end{array}\\\hline
\sqrt{-3}/2& \begin{array}{c}
1417905000-818626500\sqrt3\\
\end{array}\\\hline
\sqrt{-3}/4& \begin{array}{c}
-1160733998424384000\sqrt 3+2010450259344609000\\
+820762881440077125\sqrt 6-1421603011620136125\sqrt 2
\end{array}
\end{array}
\end{gather*}

Next, using \eqref{eq:j-u}, we compute the $u(\tau)$ values as follows.
\begin{gather*}
\rowcolors{1}{white}{pink}
\arraycolsep=4.7pt\def\arraystretch{1.2}
\begin{array}{c|cc}
\tau & u(\tau)=(\eta(2\tau)/\eta(\tau))^{24}&\\ \hline\hline
\sqrt{-3}& 13/512-15\sqrt3/1024&\\\hline
2\sqrt{-3}& \begin{array}{c}
-15\sqrt3/8-1667/512+5445\sqrt6/4096+9405\sqrt2/4096
\end{array}\\\hline
\sqrt{-3}/2&
13/2-15\sqrt3/4&\\\hline
\sqrt{-3}/4&
282645\sqrt 6/4-99930\sqrt3
-489555\sqrt2/4+173084
\end{array}
\end{gather*}

Hence we obtain that
\begin{gather*}\allowdisplaybreaks
\eta\big({-}6/\sqrt{-3}\big)=\eta(2\sqrt{-3})=\(\frac{26-15\sqrt3}{1024}\)^{1/24}\eta\big(\sqrt{-3}\big) =\frac{(3\sqrt3-5)^{1/12}}{2^{11/24}}\eta\big(\sqrt{-3}\big) ,\\
\eta\big({-}12/\sqrt{-3}\big) = \eta\big(4\sqrt{-3}\big) =\sqrt 2\big(5445 \sqrt 6+9405 \sqrt2-7680\sqrt3-13336\big)^{1/24}\eta\big(2\sqrt {-3}\big)\\
\hphantom{\eta\big({-}12/\sqrt{-3}\big)}{} = \frac{(495\sqrt6-495\sqrt2+360\sqrt3-1136)^{1/24}}{2^{11/12}}\eta\big(\sqrt{-3}\big),\\
\eta\big({-}2/\sqrt{-3}\big) =\eta\big({-}1/\big(\sqrt{-3}/2\big)\big)=\frac{3^{1/4}}{\sqrt 2}\eta\big(\sqrt{-3}/2\big)
 =\frac{3^{1/4}}{\sqrt 2}\big( 13/2-15\sqrt3/4\big)^{-1/24}\eta\big(\sqrt{-3}\big)\\
\hphantom{\eta\big({-}2/\sqrt{-3}\big)}{} =\frac{3^{1/4}}{\sqrt 2}\big(104+60\sqrt3\big)^{1/24}\eta\big(\sqrt{-3}\big), \\
\eta\big({-}4/\sqrt{-3}\big) =\eta\big({-}1/\big(\sqrt{-3}/4\big)\big)=\frac{3^{1/4}}{2}\eta\big(\sqrt{-3}/4\big)\\
\hphantom{\eta\big({-}4/\sqrt{-3}\big)}{} =\frac{3^{1/4}}2\big(282645\sqrt 6/4-99930\sqrt3 -489555\sqrt2/4+173084\big)^{-1/24}\eta\big(\sqrt{-3}/2\big)\\
\hphantom{\eta\big({-}4/\sqrt{-3}\big)}{} =\frac{3^{1/4}}{2}\big(5445\sqrt 6+7680\sqrt3-9405\sqrt2-13336\big)^{1/24}\eta\big(\sqrt{-3}/2\big) \\
\hphantom{\eta\big({-}4/\sqrt{-3}\big)}{} =\frac{3^{1/4}}{2}\big(1980\sqrt 2-4544+1980\sqrt 6-1440\sqrt3\big)^{1/24}\eta\big(\sqrt{-3}\big).
\end{gather*}

To determine the value $t\big({-}1/\sqrt {-3}\big)$, we notice that
\begin{gather}\label{eqn:jt}
j=\frac{(t+3)^3(t+9)^3\big(t^2+27\big)^3}{t^3\big(t^2+9t+27\big)^3}
\end{gather}
and $j\big({-}1/\sqrt{-3}\big)=j\big(\sqrt{-3}\big)=54000$.

This gives the information on the algebraic value $t\big({-}1/\sqrt {-3}\big)$. The relationship between~$j$ and~$t$ in~(\ref{eqn:jt}) follows Sebbar's construction in \cite[Tables~3 and~7]{Sebbar} by a variable change $t\mapsto t+3$.
\end{proof}

We end with listing the first few computed constants $C_{\chi,k}$ for the identities $L\big(\chi^k,k/2\big)=C_{\chi,k}L(\chi,1)^k$ in Table~\ref{table:1}, in which $\mathcal G_{k,(1;3)}^\ast$ is abbreviated by $\mathcal G_k^\ast$.

\begin{table}[h!]\centering
\rowcolors{1}{white}{pink}
\arraycolsep=4.7pt\def\arraystretch{1.7}
\begin{tabular}{l l l}
$ L\big(\chi^4,2\big)=\frac 92 L(\chi,1/2)^4$ & $\mathcal G_4^*= \frac{(-2\pi {\rm i})^4}{3!}\phi\phi_1^3$\\
$ L\big(\chi^5, 5/2\big)=6L(\chi,1/2)^5$& $\mathcal G_5^*= \frac{(-2\pi {\rm i})^5}{4!}\phi^2\phi_1^3$\\
$ L\big(\chi^6,3\big)= \frac{288}{35} L(\chi,1/2)^6$& $\mathcal G_6^*=\frac{(-2\pi {\rm i})^6}{5!}\big(\phi^3\phi_1^3+12\phi_1^6\big)$ \\
$ L\big(\chi^7, 7/2\big)=12L(\chi,1/2)^7$ &$ \mathcal G_7^*=\frac{(-2\pi {\rm i})^7}{6!}\big(\phi^4\phi_1^3+36\phi\phi_1^6\big)$ \\
$ L\big(\chi^8,4\big)= \frac{243}{14}L(\chi,1/2)^8$ & $\mathcal G_8^*=\frac{(-2\pi {\rm i})^8}{7!}\big(\phi^5\phi_1^3+96\phi^2\phi_1^6\big)$ \\
$ L\big(\chi^9,9/2\big)=\frac{512}{21}L(\chi,1/2)^9$ & $\mathcal G_9^*=\frac{(-2\pi {\rm i})^9}{8!}\big(\phi^6\phi_1^3+216\phi^3\phi_1^6+720\phi_1^9\big) $ \\
$ L\big(\chi^{10},5\big)=\frac{243}7L(\chi,1/2)^{10}$ & $\mathcal G_{10}^*=\frac{(-2\pi {\rm i})^{10}}{9!}\big(\phi^7\phi_1^3+468\phi^4\phi_1^6+4752\phi\phi_1^9\big) $\\
$ L\big(\chi^{11},{11}/2\big)=\frac{348}7L(\chi,1/2)^{11}$ & $\mathcal G_{11}^*= \frac{(-2\pi {\rm i})^{11}}{10!}\big(\phi^8\phi_1^3+972\phi^5\phi_1^6+22896\phi^2\phi_1^9\big)$
\end{tabular}
\caption{List of $C_{\chi,k}$.}\label{table:1}
\end{table}

It is natural to seek the arithmetic meaning of these numbers $C_{\chi,k}$. In terms of congruences among modular forms $h_k(\tau) = \sum\limits_{n\ge 1}a_{k}(n)q^n$ with $L(h_k, s) = L\big(\chi^{k-1}, s - \frac{k-1}{2}\big)$, we discovered the following congruence relations numerically:
\begin{gather*}
a_3(n)\equiv a_2(n) \mod 3,\\
a_4(n)\equiv a_2(n) \mod 8,\\
a_5(n)\equiv a_2(n) \mod 9,\\
a_6(n)\equiv a_2(n) \mod 240.
\end{gather*}The first one can be proved immediately since for $n=p\equiv 1\mod 3$, $a_3(p)= a_2(p)^2-2p$ and $a_2(p)\equiv -1\mod 3$. Likewise, the other congruences can be obtained case by case from the expression
\begin{gather*}
a_{k+1}(p)=\big(a+b\sqrt{-3}\big)^k+\big(a-b\sqrt{-3}\big)^k, \qquad a\equiv 1 \mod 3,
\end{gather*}
for primes $p\equiv 1\mod 3$. Deeper relations are subject to future investigation.

\subsection*{Acknowledgments}
The research of Li is partially supported by Simons Foundation grant \# 355798, Long by NSF grant DMS \#1602047, and Tu by the AWM-NSF Mentoring Travel Grant which supported her visit to Pennsylvania State University to work with Li in 2017 {and 2018}. The authors would like to thank the anonymous referees, Jerome W.~Hoffman, Ming-Lun Hsieh, Ye Tian, Yifan Yang, and Wadim Zudilin for their helpful comments and feedback. At the early stage of this work the authors benefited from their visits to the Institute of Mathematical Research in the University of Hong Kong in summer 2017. To IMR they express their gratitude for its hospitality.

\pdfbookmark[1]{References}{ref}
\LastPageEnding

\end{document}